\documentclass[11pt,reqno]{amsart}
\usepackage{amsmath,amsthm,amssymb,amsfonts,enumerate,color,enumerate}
\newcommand{\mysection}[1]{\section{#1}
      \setcounter{equation}{0}}

\newcommand\cbrk{\text{$]$\kern-.15em$]$}} 
\newcommand\opar{\text{\raise.2ex\hbox{${\scriptstyle | }$}\kern-.34em$($} }
\newcommand\loc{\rm loc\,}
\newcommand\sca{\text{\sc a}}
\newcommand\scb{\text{\sc b}}
\newcommand\scc{\text{\sc c}}

\usepackage{color}

\DeclareMathOperator*{\esssup}{ess\,sup}

\newtheorem{theorem}{Theorem}[section]
\newtheorem{lemma}[theorem]{Lemma}
\newtheorem{proposition}[theorem]{Proposition}
\newtheorem{corollary}[theorem]{Corollary}

\theoremstyle{definition}
\newtheorem{assumption}{Assumption}[section]
\newtheorem{definition}{Definition}[section]

\theoremstyle{remark}
\newtheorem{remark}{Remark}[section]

\newcommand\bL{\mathbb{L}}
\newcommand\bR{\mathbb{R}}

\newcommand\bB{\mathbb{B}}

\newcommand\bT{\mathbb{T}}

\newcommand\bW{\mathbb{W}}

\newcommand\frb{\mathfrak{b}}

\newcommand\frf{\mathfrak{f}}

\newcommand\frl{\mathfrak{l}}
\newcommand\fru{\mathfrak{u}}

\newcommand\frF{\mathfrak{F}}

\newcommand\frH{\mathfrak{H}}
\newcommand\frG{\mathfrak{G}}

\newcommand\cB{\mathcal{B}}
\newcommand\cD{\mathcal{D}}

\newcommand\cF{\mathcal{F}}

\newcommand\cH{\mathcal{H}}

\newcommand\cP{\mathcal{P}}

\newcommand\cR{\mathcal{R}}
\newcommand\cS{\mathcal{S}}

\newcommand\cV{\mathcal{V}}

\newcommand\cW{\mathcal{W}}

\makeatletter
 \newcommand{\sumstar}%
 {\operatornamewithlimits{\sum@\kern-.2em\raise1ex\hbox{*}}}
 \makeatother

\def\Xint#1{\mathchoice
{\XXint\displaystyle\textstyle{#1}}%
{\XXint\textstyle\scriptstyle{#1}}%
{\XXint\scriptstyle\scriptscriptstyle{#1}}%
{\XXint\scriptscriptstyle\scriptscriptstyle{#1}}%
\!\int}
\def\XXint#1#2#3{{\setbox0=\hbox{$#1{#2#3}{\int}$ }
\vcenter{\hbox{$#2#3$ }}\kern-.6\wd0}}

\def\dashint{\Xint-}

\begin{document}
 
\title[On Stochastic Navier-Stokes Equations] 
{On conditional uniqueness of solutions to stochastic Navier-Stokes Equations}

\author[I. Gy\"ongy]{Istv\'an Gy\"ongy}
\address{School of Mathematics and Maxwell Institute,
University of Edinburgh,
King's  Buildings,
Edinburgh, EH9 3JZ, United Kingdom}
\email{gyongy@maths.ed.ac.uk}

\author[N.V. Krylov]{Nicolai V. Krylov}%
\thanks{}
\address{127 Vincent Hall, University of Minnesota,
Minneapolis,
       MN, 55455, USA}
\email{krylov@math.umn.edu}

\subjclass[2020] {35Q30, 35R60, 60H15}
\keywords{ Navier-Stokes equations, stochastic Navier-Stokes equations, 
weak solutions, conditional 
uniqueness and regularity, Ladyzhenskaya-Prodi-Serrin condition}

\begin{abstract}
 This paper is a continuation of \cite{GK2024}. Here theorems on 
conditional uniqueness and regularity for solutions to stochastic Navier-Stokes 
equations in $\bR^d$ are presented. 
 
\end{abstract}

\maketitle

\mysection{Introduction}

The Navier-Stokes equations, 
\begin{equation} 
                                                                  \label{eq0}
\frac{\partial}{\partial t}u=\nu\Delta -u_{(u)}-\nabla p+D_j\frf_j+f, 
\quad 
\text{div}\, u=0
\end{equation}
for $t\geq 0$ and $x\in\bR^d$, 
with initial condition 
\begin{equation}                                        \label{ini0}
u_{t}\big|_{t=0}=u_{0}
\end{equation}
for the evolution of the velocity and pressure fields 
$u=(u^{1}_t(x),...,u^{d}_t(x))$ 
and $p=p_t(x)$,  
are among the most important equations 
in fluid dynamics 
and are also among the most studied PDEs in the literature.   
Here $\nu$ is a positive constant, 
$u_{(u)}=u^jD_ju$,   
$$
D_j\frf_j
=(\tfrac{\partial}{\partial x^j}\frf^1_j,....,\tfrac{\partial}{\partial x^j}\frf^d_j), 
\quad 
f=(f^1,...,f^d)
$$
are given 
force fields for $t\geq0$ and $x=(x^1,...,x^d)\in\bR^d$, and 
$u_0=(u_0^1,...,u_0^d)$
is a given velocity field on $\bR^d$
such that $\text{div}\, \frf_{j}=\text{div}\, u_{0}=0$.

By the classical results of Hopf \cite{H1951} and Leray \cite{L1934} 
for any $u_{0}\in \cH$ 
at least one weak solution to \eqref{eq0}-\eqref{ini0} exists in the {\it Hopf-Leray 
class}, 
$$
\cW:=L_{\infty}([0,\infty),\cH)\cap L_2([0,\infty),\cV)), 
$$
where 
$$
\cV:=\{u\in W^1_2(\bR^d,\bR^d): {\rm div}\, u=0\}  
$$
and $\cH$ is the closure of $\cV$ in $L_2(\bR^d,\bR^d)$. 
(See, e.g.,  \cite{L1963}, \cite{LSU1967}, \cite{L1969}, \cite{T1977},  
\cite{F2021} and \cite{R2022} for presentations of Leray and Hopf results 
and for further developments.)

Due to a famous theorem of Ladyzhenskaya, the uniqueness of the 
 weak solution in this class is known when $d=2$, and it is an open problem 
for $d\geq3$. There is, however a considerable literature on {\it 
conditional uniqueness} of the solution under various criteria. 
The most well-known criterium is the Ladyzhenskaya-Prodi-Serrin 
condition, which requires that for a $T>0$ the Hopf-Leray weak 
solution be in the space $L_{p,q}=L_{p,q}([0,T]\times\bR^d, \bR^d)$ 
such that 
\begin{equation} 
                                                                                \label{LPS0}
\frac{d}{p}+\frac{2}{q}\leq 1\quad 
 \text{for some $p\in(d,\infty]$ and $q\in[2,\infty)$}, 
\end{equation}
where $L_{p,q}$  denotes the space 
of $\bR^d$-valued functions $v$ on $[0,T]\times\bR^d$ 
such that 
$$
|v|^q_{L_{p,q}}:=\int_0^T|v_t|^q_{L_p}\,dt<\infty. 
$$
Under this condition the uniqueness of the Hopf-Leray 
weak solution on $[0,T]$ was proved by 
Prodi \cite{Prodi1959} and  Serrin \cite{Serrin1962}, 
and smoothness of the solution was obtained by 
Ladyzhenskaya \cite{L1967}.   

Our aim in the present paper is to generalise the conditional 
uniqueness result of Prodi and Serrin and to extend it 
stochastic Navier-Stokes equations 
and to prove also a theorem on conditional regularity of 
their solutions. Our condition
is much weaker than the Ladyzhenskaya-Prodi-Serrin 
condition (see Corollary \ref{corollary LPS}). 

We will consider the following type of stochastic Navier-Stokes equations,  
\begin{align}
du=&\big(\nu\Delta u-u_{(u)}
+\nabla p+D_j\frf_j +f \big)\,dt                                                                  \nonumber\\
&+\big(\sigma\nabla u-\nabla q+h\big)\circ dw, 
\quad
 \text{div}\, u=0,                                                                          \label{NS}\\
u_{t}\big|_{t=0}=&u_{0}, \quad                                                        \label{NSi}
\end{align}
for the random velocity field $u=(u_t^1(x),...,u^d_t(x))$,  
and pressure fields $p=p_t(x)$ 
and $q=q_t(x)$ for $t\in[0,T]$ and $x\in\bR^d$, where $w$ is a Hilbert space-valued 
Wiener process, and $\circ dw_t$ indicates that the corresponding differential 
is understood in the Stratonovich sense. As it was shown 
in \cite{MR2001} and 
\cite{MR2004}, equation \eqref{NS} arises when one follows the classical scheme 
of Newtonian fluid dynamics, assuming that instead of a deterministic dynamics, 
the fluid particles satisfy the stochastic equation
$$
d\eta_t(x)=u_t(\eta_t(x))\,dt
+\sigma_t(\eta_t(x))\circ dw_t, \quad \eta_0(x)=x, 
\quad x\in \bR^d. 
$$ 
In \cite{MR2005} the authors study the Cauchy problem \eqref{NS}-\eqref{NSi} 
with a second order elliptic differential operator instead of $\Delta$,  and with force fields 
$D_j\frf_j$, $f$ and $h$, which may depend also on $u$. 
They show the existence of a solution on some filtered probability space $(\Omega,\cF,(\cF_t)_{t\geq0},P)$ 
carrying an $\cF_t$-Wiener process $w$ with values in a Hilbert space. 
Moreover, in $d=2$ they show the existence 
and uniqueness of a unique strong solution.
The solution $u=(u_t)_{t\in[0,T]}$ in \cite{MR2005} is understood as an $\cH$-valued 
weakly continuous $\cF_t$-adapted process such that $u\in L_2([0,T], \cV)$ (a.s.), 
and almost surely satisfies the equations in a weak sense, similarly 
to the definition of a Hopf-Leray weak solution to \eqref{eq0}-\eqref{ini0}.  

We consider for $d\geq3$ the type of stochastic Navier-Stokes equations 
as in \cite{MR2005}. Adapting the notion of admissible random functions 
from \cite{K2022b} and \cite{{GK2024}}, we introduce the class of admissible 
solutions to them, see Definitions \ref{definition HL} 
and \ref{definition admissible}. Roughly speaking 
an admissible function $u$ is function of $(\omega, t,x)$, 
such that it can be decomposed 
as a sum of two functions, say $u^M$ and $u^B$, 
such that $u^M$, as a function of $x\in\bR^d$, belongs to 
the Morrey space $E_{r,1}$ for an $r\in(2,d]$ and   
its Morrey norm is bounded by a constant, uniformly in 
$(t,\omega)$. The other component, $u^B$ belongs 
to $L_2\big([0,T], L_{\infty}(\bR^d,\bR^d)\big)$ (a.s.). 

First we show that an admissible solution has a 
modification which is a (strongly) continuous $\cH$-valued 
function in $t\in [0,T]$ for every $\omega\in\Omega$. This is 
Theorem \ref{theorem 1} below. Then we present a theorem, 
Theorem \ref{theorem 2} on conditional uniqueness, which in the special case 
of equation \eqref{NS} reads as follows. 

{\em
There is a positive constant $N$ depending only $d$, $r$, such that 
if $u$ is an admissible solution with Morrey norm bounded by $\nu/N$, 
then every admissible solution coincides with $u$.} 

Hence, see 
Corollary \ref{corollary LPS},  we get that if $u$ is a solution 
such that $u\in L_{p,q}$ almost surely for a $(p,q)$ satisfying 
the Ladyzhenskaya-Prodi-Serrin condition \eqref{LPS0}, then 
the uniqueness holds in the class of admissible solutions. 
This is an immediate consequence of Theorem \ref{theorem 2}, 
since it turns out that if $u $ is a function from 
$L_{p,q}$ (a.s.) such the condition \eqref{LPS0} 
holds then $u$ is admissible and admits a Morrey component 
with Morrey norm bounded by a positive constant as small as we wish. 
It is worth noting that the admissible functions form an essentially larger 
space than $L_{p,q}$ with \eqref{LPS0}. 
Thus Theorem \ref{theorem 2} generalises the result of Prodi and Serrin 
in the special case of deterministic Navier-Stokes equations, i.e., 
when $\sigma=0$, $q=0$ and $h=0$, and extends that also to the 
type of stochastic Navier-Stokes equations \eqref{NS}-\eqref{NSi}. 
Under the conditions of Theorem \ref{theorem 2} we present 
also a result, Theorem \ref{theorem 3}, 
on conditional regularity of the solution, when $u_0\in\cV$ (a.s.),   
$\frf=0$, and $\sigma$ and $h$ satisfy appropriate differentiability 
conditions. 
We will prove further results on conditional regularity 
in the continuation of the present paper. 

Stochastic Navier-Stokes equations are also extensively studied 
in the literature. We mention here only a  few publications 
concerning their theory, starting with the pioneering paper  
\cite{BT1973}, followed by \cite{AC1990}, \cite{CC1991}, \cite{FY1992}
\cite{BCF1992}, \cite{CG1994}, \cite{B1995}, \cite{DD2002} and others  
 on solvability of the equations in various set-ups and approaches. 
 Uniqueness results in $d=2$ can be found, e.g., in 
 \cite{AF2002}, \cite{MS2002} \cite{AF2004}, \cite{F2003}, \cite{MR2005}, 
 \cite{M2009}. 
 Existence of local $L_p$-solutions 
 to Navier-Stokes equations with L\'evy noise is studied in \cite{MS2017}. 
 The existence and uniqueness of solutions in an $L_p$ setting 
 for 2D stochastic Navier-Stokes  with jump noise is investigated in 
 \cite{ZBL2019}, the local well-posedness in $L_p$-setting 
 for 3D  Navier-Stokes equations with multiplicative cylindrical noise in the whole space is studied  in \cite{KWX2024}, 

Existence and uniqueness of a global mild solution to random vorticity 
equations associated to stochastic Navier-Stokes equations in $d=3$ 
for small initial vorticity is obtained  in \cite{BR2017}. 
Theorems on ergodicity in $d=2$ are presented in 
\cite{WMS2001}, \cite{WM2001}, \cite{KS2000}, \cite{KS2001}, 
\cite{B2002}, \cite{MH2006},  
and a result on ergodicity in d=3 is obtained  \cite{DD2003}. 
Balance relations for randomly forced 
Navier-Stokes equations on the 2-dimensional torus are proved 
in \cite{KP2005}.  
In \cite{AV2024} well-posedness, regularisation results for the solutions and blow-up criteria 
are established, and a global well-posedness result with high probability for small initial data, in critical spaces is proved for stochastic Navier--Stokes equations on the d-dimensional torus for any $d\geq2$.

For a comprehensive 
treatise on stochastic fluid dynamics we refer to \cite{F2008}.

We conclude with notations used in the paper. 
Let $C_0^{\infty}(\bR^d,\bR^d)$ be the space of $\bR^d$-valued 
compactly supported smooth functions on $\bR^d$ and let  
$\cD$ denote its subspace consisting of the divergence-free functions.
For integers $m\geq0$ and $p\geq1$ we use the notation 
$\bW^m_p=W^m_p(\bR^d,\bR^d)$   
for the Sobolev space of 
$\bR^d$-valued functions $v=(v^1,...,v^d)$ on $\bR^d$, which together 
with their generalised 
derivatives up to order  $m$ are $L_p$-functions. 
We use also the notation 
$\bL_p$ for $\bW^0_p$. 
Note that $\bW^m_2$ 
with the inner product 
$$
(u,v)_{\bW^m_2}:=\sum_{|\alpha|\leq m}^m(D^{\alpha}u,D^{\alpha}u), 
\quad u,v\in \bW^m_2
$$
is a separable Hilbert space, where for any integer $k\geq1$ 
for $\bR^k$-valued functions 
$v$ and $w$ on $\bR^d$ we use the notation 
$$
(v,w)=\int_{\bR^d}v^i(x)w^i(x)\,dx 
$$
when the  integral is well-defined. 
If $B$ is a ball in $\bR^{d}$ by $|B|$ we mean its volume
and set
$$
\dashint_{B}f\,dx=\frac{1}{|B|}\int_{B}f\,dx.
$$        
The derivatives of $Dv$ of $\bR^m$-valued functions in $x\in\bR^d$ we view 
as functions with values in the $m\times d$ matrices,  $(D_iv^n)=(v^{ni})$, 
and we treat the space of $m\times d$ matrices as the $md$-dimensional 
Euclidean space as well. For functions $f$ on $\bR^d$, 
with values in a tensor space $\bT$,  
the notation $|f|_{L_2}$ is defined by
$$
|f|_{L_2}^2=\int_{\bR^d}|f(x)|^2_{\bT}\,dx,
$$
where $|f(x)|_{\bT}$ denotes the Hilbert-Schmidt norm of the tensor $f(x)$. 

Let 
$\cH$ and $\cV$ be the divergence free subspaces of 
$\bL_2:=\bW^0_2$ and $\bW^1_2$, respectively, i.e.,
$\cH$ is the closure of $\cD$ in $\bL_{2}$, and 
$$
\cV:=\{u\in\bW^1_2: {\rm div}\, u=0\}.
$$
Let $\cS$ denote the orthogonal projection 
of $\bL_2$ onto $\cH$, and define the operator $\cR=I-\cS$, 
where $I$ is the identity operator.  

For functions $\varphi$ 
and vector fields $u=(u^1,...,u^d)$ on $\bR^d$ 
we use the notations 
$$
D_i\varphi=\frac{\partial}{\partial x^i}\varphi, \,\,(i=1,2,...,d), \quad
\nabla\varphi =D\varphi=(D_1\varphi,...,D_d\varphi), 
$$
$$
\varphi_{(u)}=u^jD_j\varphi,
$$
 $Du$ for the Jacobian matrix $(Du)^{ij}=D_ju^i$,  
and $D^2u$ for the tensor $(D^2u)^{ijl}=D_{i}D_{j}u^l$, $i,j,l=1,2,...,d$.
Unless otherwise stated, thorough the paper  
we use the summation convention 
with respect to 
repeated integer-valued indices. Thus $u_{(u)}$ 
denotes the $\bR^d$-valued 
function 
$$
(u^jD_ju^1,u^jD_ju^2,...,u^jD_ju^d)
$$
on $\bR^d$. 
For normed spaces $B$ we use the notation 
$\ell_2(B)$ for the space of sequences $c=(c_n)_{n=1}^{\infty}$ of 
elements $c_n\in B$, with the norm 
$|c|_{\ell_2(B)}=(\sum_n|c_n|^2_{B})^{1/2}$. We write 
$\ell_2$ instead of $\ell_2(\bR)$. 

All random elements in this paper are defined on a fixed   
filtered probability space $(\Omega,\cF,(\cF_t)_{t\geq0}, P)$ 
carrying a sequence of independent $\cF_t$-Wiener 
processes $(w^k)_{k=1}^{\infty}$. We assume that $\cF$ is $P$-complete 
and that $\cF_0$ contains all the $P$-zero sets of $\cF$. We use the notation 
$\cP$ for the predictable $\sigma$-algebra on $\Omega\times[0,\infty)$.
For a topological space $S$ the notation $\cB(S)$ means the Borel 
$\sigma$-algebra on $S$.

\mysection{Formulation of the main theorems}
                                                                         \label{section formulation}
Consider the equations                                                                          
\begin{align} 
                                                                  \label{eq}
du_t=&\big(D_i(a^{ij}_{t}D_{j}u_t+\frf_t^i(u_t))
+f_t(u_t,Du_t)
-u_{t(u_t)}-\nabla p_t+\gamma^{ki}_tD_{i}q_t^k\big)\,dt                                 \nonumber\\
&+(\sigma^{ik}_{t}D_{i}u_t+h_t^{k}(u_t)
-\nabla q^k_t)\,dw^k_t, 
\quad 
\text{div}\, u_t=0 
\end{align} 
on $\Omega\times[0,T]\times\bR^d$, with initial condition 
\begin{equation}                                       
                                                             \label{ini}
u(0,x)=u_0(x),\quad x\in\bR^d  
\end{equation}
for a velocity field 
$u=(u_t^{1}(t,x),...,u^{d}(t,x))$ 
and pressure fields $p=p_t(x)$ 
and $q=(q^k_t(x))_{k=1}^{\infty}$ on $\Omega\times[0,T]\times\bR^d$. 

We assume that

(i)
the coefficient  $a=(a^{kl})$ is $\bR^{d\times d}$-valued, 
$\gamma^i=(\gamma^{ik})_{k=1}^{\infty}$ is an 
$\ell_2(\bR^d)$-valued  
and $\sigma^i=(\sigma^{ik})_{k=1}^{\infty}$ 
is an $\ell_2$-valued $\cP\otimes\cB(\bR^d)$ measurable function  
on $\Omega\times[0,T]\times\bR^d$ for each $i=1,2,...,d$; 

(ii) the functions $\frf^i$ and $f$ are $\bR^d$-valued, $\frf^i$ is   
a $\cP\otimes\cB(\bR^d\times\bR^d)$-measurable mapping on 
$\Omega\times[0,T]\times\bR^{d}\times\bR^d$, 
and $f$ is a 
$\cP\otimes\cB(\bR^d\times\bR^d\times\bR^{d\times d})$-measurable 
mapping on 
$\Omega\times[0,T]\times\bR^d\times\bR^{d}\times\bR^{d\times d}$, 
for each $i=1,2,...,d$.  
The function $h=(h^k)_{k=1}^{\infty}$ is an $\ell_2(\bR^d)$-valued 
$\cP\otimes\cB(\bR^d\times\bR^d)$-measurable mapping on 
$\Omega\times[0,T]\times\bR^d\times\bR^d$.

\begin{definition}
                                                                         \label{definition HL}
A $\cV$-valued predictable process $u=(u_t)_{t\in[0,T]}$ is a solution to 
\eqref{eq}-\eqref{ini}  if 
$u\in L_2([0,T],\cV)\cap L_{\infty}([0,T],\cH)$ (a.s.),   
and for each $v\in \cV \cap \bL_{ d} $ 

\begin{align}
(u_t,v)=&(u_0,v)
 -\int_0^t\big[(a_s^{ij}D_ju_s+\frf^i_s(u_s),D_iv)
 +(u_{s(u_s)},v)\big]\,ds                                                      \nonumber\\
&+\int_0^t (\gamma^{ki}_sD_iq^k_s
+f_s(u_s,Du_s), v)\,ds      \nonumber\\
 &+\int_0^t(\sigma_s^{jk}D_ju_s+h^k_s(u_s),v)\,dw^k_s,                    \label{sol1a}
 \end{align}
 holds for $P\otimes dt$-almost every $(\omega,t)\in\Omega\times[0,T]$, 
where 
 \begin{equation}
                                                                                         \label{sol1b}
 \nabla q^k_t=\cR(\sigma_t^{jk}D_ju_t+h^k_t(u_t))
 \quad 
 \text{$P\otimes dt\otimes dx$-a.e. \,\, for every $k\geq1$}.                                                                                     
\end{equation}
 We call $(u_t)_{t\in[0,T]}$ an $\cH$-solution if it is an 
 $\cH$-valued $\cF_t$-adapted continuous process 
 such that $u\in L_2([0,T], \cV)$ (a.s.), and 
 almost surely \eqref{sol1a} holds 
 for all $t\in[0,T]$ and $v\in \cV\cap \bL_{ d} $,  and
  \eqref{sol1b} is satisfied. 
   \end{definition} 
 \begin{remark}
                                                                                        \label{remark q}
 We note that if Assumption \ref{assumption 1} and 
 \eqref{h} below hold for $h$, then equation \eqref{sol1b} ensures that for any solution $u$
 $$
 M^k_t(u_t)-\nabla q^k_t=\cS M^k_t(u_t)\in\cH\quad \text{($P\otimes dt\otimes dx$-a.e), 
 \quad for $k=1,2,...$}, 
 $$ 
where 
\begin{equation}
                                                                                          \label{M}
M^k_t(u):=\sigma_t^{ik}D_iu+g^k_t(u) \quad\text{for $u\in\cV$}. 
\end{equation}
 Moreover, by equation \eqref{sol1b} we have 
\begin{equation}
                                                                                          \label{q}
\gamma^{ki}D_iq^k=\gamma^{ki}(\cR M^k(u))^i
\quad
\text{for each $k\geq1$},
\end{equation}
where $(\cR M^k(u))^i$ is the $i$-th coordinate of $\cR M^k(u)$, 
which closes the equation \eqref{sol1a} for $u$.
\end{remark}  
 \begin{remark}
If $q>1$ such that $\tfrac{1}{q}=\tfrac{1}{2}-\tfrac{1}{d}$, 
 then $\tfrac{1}{q}+\tfrac{1}{2} +\tfrac{1}{d}=1$ 
 and thus by H\"older's inequality and 
 by the Sobolev inequality $|w|_{L_q}\leq N(d)|w|_{W^1_2}$ 
we have 
 $$
 |(v_{(u)},w)|\leq |u|_{L_q}|Dv|_{\bL_2}|w|_{\bL_d}
 \leq N(d)|u|_{\bW^1_2}|Dv|_{\bL_2}|w|_{\bL_d}<\infty
 $$
for any $u,v,w\in \bW^1_2$. 
Hence for $(u_{s(u_s)},v)$ in \eqref{sol1a} we have 
$$
\int_0^T|(u_{s(u_s)},v)|\,ds\leq N(d)|v|_{\bL_d}\int_0^T|u_s|^2_{\bW^1_2}\,ds<\infty 
\,\text{(a.s.) for any $v\in\bL_d$}.
$$
This is why test functions $v$ from $\cV\cap L_d$ are used. 
 \end{remark} 
 
We  make the following assumptions. 
\begin{assumption}                                                      \label{assumption 1} 
(i) There is a constant $\delta>0$ such that 
\begin{equation}
                            \label{4.30.1}
|a|\leq \delta^{-1},
\quad
\big(a^{ij}-\tfrac{1}{2}\sigma^{ik}\sigma^{jk}\big)
\lambda^i\lambda^j\geq \delta |\lambda|^2
\quad
\text{for all $\lambda\in\bR^d$},\omega,t,x.
\end{equation}
(ii) There is a nonnegative predictable process $\vartheta_t$, $t\in[0,T]$, such that 
\begin{equation*}
\int_0^T\vartheta_t^2\,dt<\infty
\end{equation*}
and 
\begin{equation}                                                                                                                                                                                                                                                                        
                                                                                 \label{gamma}                                
|\gamma_t(x)|^2:=\sum_{i=1}^d\sum_{k=1}^{\infty}|\gamma_t^{ki}(x)|_{\bR^d}^2
\leq \vartheta_t^2        \quad\text{for all $\omega,t,x$}.
\end{equation}
\end{assumption}
\begin{remark}
                                                                    \label{remark sigma}
Note that by Assumption \ref{assumption 1}\,(i) we have 
$$
|\sigma_t|^2=\sum_{i=1}^d\sum_{k=1}^{\infty}|\sigma^{ik}|^2\leq N^2
$$ 
for a nonnegative constant $N=N(d,\delta)$. 
\end{remark}
\begin{assumption}
                                                      \label{assumption 2} 
There exist nonnegative predictable processes  
$\lambda$, $\kappa$ and $\chi$, and nonnegative 
$\cP\otimes\cB(\bR^d)$-measurable functions $\frF$ 
and $\frG$ on 
$\Omega\times[0,T]\otimes\bR^d$ such that 
\begin{equation}
                                                            \label{lambda}
\int_0^T(\lambda_t^2+\chi_t^2+\kappa_t)\,dt<\infty, \,\,
\int_0^T(|\frF_t|^2_{L_2}+|\frG_t|_{L_2})
\,dt<\infty 
\,\,\,(a.s.), 
\end{equation}
and

\noindent
(i) for all $\omega\in\Omega$, $t\in[0,T]$, 
$x,u\in\bR^d$ and $z\in\bR^{d\times d}$
\begin{equation}
                                                                                        \label{h}
|\frf_t(x,u)|+|h_t(x,u)|\leq \lambda_t|u|+\frF_t(x), 
\end{equation}
$$
|f_t(x,u,z)|\leq \kappa_t|u|+\chi_t |z|+\frG_t(x), 
$$
(ii) for all $\omega\in\Omega$, $t\in[0,T]$, $x$, 
$u_i\in\bR^d$ and $z_i\in\bR^{d\times d}$ 
($i=1,2$) 
\begin{equation}                                                              \label{dh}
|\frf_t(x,u_1)-\frf_t(x,u_2)|+|h_t(x,u_1)-h_t(x,u_2)|
\leq\lambda_t|u_1-u_2|, 
\end{equation}
\begin{equation}                                                              \label{f}
|f_t(x,u_1,z_1)-f_t(x,u_2,z_2)|
\leq \kappa_t|u_1-u_2|+\chi_t|z_1-z_2|.  
\end{equation}
\end{assumption}
\begin{assumption} 
                                                     \label{assumption 3}
The initial condition $u_{0}$ is an 
$\cF_0$-measurable random variable in $\cH$. 
\end{assumption}

We assume $d\geq 3$. 
To formulate our results we define the following class of functions 
for a fixed $\rho_0\in(0,1]$, where we use    
the notation $\bB_{\rho}$ for the set of 
balls in $\bR^d$ with radius $\rho$. 
\begin{definition} 
                                                                         \label{definition admissible}
A real-, vector-, or tensor-valued function $f$ defined on 
$\Omega\times[0,T]\times\bR^d$ is called admissible  if 
$f=f^M+f^B$ for some $\cP\otimes\cB(\bR^d)$-measurable 
functions $f^M$ and $f^B$, such that for some $r\in(2,d]$  
with a  {\em constant\/} $\hat f$ and a nonnegative predictable 
function $\bar f$ on $\Omega\times[0,T]$ 
we have 
\begin{equation} 
                                         \label{morrey}
\left(\dashint_{B_{\rho}}|f^M(t,x)|^{r}
\,dx\right)^{1/r}\leq \hat f\rho^{-1}
\quad\text{for $B_{\rho}\in\bB_{\rho}$ and $\rho\leq \rho_0$}, 
\end{equation}
for all $(\omega,t)\in\Omega\times[0,T]$, 
\begin{equation}
                                                              \label{B}
\esssup_{\bR^d} | f_t^B(x)|\leq \bar f_t 
\quad \text{for all $(t,\omega)$, \,\,and  
\quad 
$\sup_{\Omega}\int_0^T \bar f^{2}_t\,dt<\infty$}.
\end{equation}
\end{definition} 
\smallskip

Observe that $f(x)=|x|^{-1}$
satisfies \eqref{morrey} with finite $\hat f$ for any $r\in[1,d)$
(and $\rho_{0}=\infty$).

For $r\in[1,\infty)$ and $\lambda\geq0$ we denote by 
$E_{r,\lambda}$ the {\it Morrey space} of functions $f$ on $\bR^d$ with values 
in a Euclidean space, such that 
$$
|f|_{r,\lambda}:=\sup_{\rho\in(0,\rho_0],B\in\bB_{\rho}}\rho^{\lambda}
\left(\dashint_{B_{\rho}}|f(x)|^r\,dx\right)^{1/r}<\infty.
$$
Then \eqref{morrey} means that $|f^M|_{r,1}\leq \hat f$.

We say that an $\bR^d$-valued function $u$ 
is an admissible solution to \eqref{eq}-\eqref{ini} if it is an admissible function 
and $(u_t)_{t\in[0,T]}$ is a solution to \eqref{eq}-\eqref{ini} in the sense of 
Definition \ref{definition HL}. 
\begin{theorem}            
                                                              \label{theorem 1}
Let Assumptions \ref{assumption 1},  \ref{assumption 2} (i)
and \ref{assumption 3} hold. Assume $u$
is an admissible  solution to \eqref{eq}-\eqref{ini}.  
Then $u$ is an $\cH$-solution in the sense that it has a 
$P\otimes dt$-modification, denoted also by $u$,  which is an $\cH$-solution. 
Moreover, 
$$
E\sup_{t\leq T}|e^{-\varphi_t}u_t|^2_{\cH}
+E\int_0^T\alpha_t|e^{-\varphi_t}u_t|^2_{\cH}\,dt 
+E\int_0^T|e^{-\varphi_t}u_t|^2_{\cV}\,dt  \leq  NE|u_0|^2_{\cH}
$$
\begin{equation}
                                                       \label{u estimate}   
 +NE \int_0^T|e^{-\varphi_t}\frF_t|^2_{L_2}\,dt
 +NE \int_0^T|e^{-\varphi_t}F_t|^2_{L_2}\,dt
 +NE\Big( \int_0^T|e^{-\varphi_t}\frG_t|_{L_2}\,dt\Big)^2,                                                                                                       
\end{equation}
where $N=N(d,\delta)$ is a constant and 
$$
\varphi_t=\int_0^t\alpha_s\,ds, 
\quad \alpha_s=N'(\lambda_s^2+\chi_s^2+\theta_s^2+\kappa_s)+\mu_s
$$
with a constant $N'=N'(\delta)$ and any 
predictable process $\mu\geq0$ such that 
$$
\int_0^T\mu_t\,dt<\infty\quad \text{for all $\omega\in\Omega$.} 
$$

\end{theorem} 
\begin{theorem}            
                                                              \label{theorem 2}
Let Assumption \ref{assumption 1},  \ref{assumption 2}  and 
\ref{assumption 3} hold, and assume that $u$ and 
$v$ are admissible  solutions
 to \eqref{eq}-\eqref{ini} on $[0,T]$.
Then there exists a (finite) constant 
$N=N(d,r)>0$ such that if 
\begin{equation}
                                                     \label{N}
\min(\hat u,\hat v)< \delta/N, 
\end{equation}
then almost surely $u(t)=v(t)$ for every $t\in[0,T]$. 
\end{theorem} 
\begin{corollary}
                                \label{corollary LPS}
Let Assumptions  \ref{assumption 1}, \ref{assumption 2}  hold, 
and assume that 
$u$ is a solution to \eqref{eq}-\eqref{ini} such that 
almost surely 
$$
\int_0^T|u_{s}|^q_{L_p}\,ds<\infty 
$$
for some $d<p\leq\infty$ and $2\leq q<\infty$ satisfying 
the Ladyzhenskaya-Prodi-Serrin 
condition 
\begin{equation} 
                                                                                  \label{LPS}
\frac{d}{p}+\frac{2}{q}\leq 1.
\end{equation}
Then $u$ is an $\cH$-solution, and if $v$ is any admissible solution 
then for its $\cH$-valued continuous modification, denoted  also by $v$, 
we have that almost surely 
$u_t=v_t$ for all $t\in[0,T]$.
\end{corollary}
 
\begin{proof}
We show (following  Remark 7.1 in \cite{GK2024}) that $u$ is admissible, with $\hat u\leq \varepsilon$ for as small $\varepsilon>0$ as 
we wish, which by virtue of condition \eqref{N} proves the corollary. 
Indeed, if $p<\infty$ then let $r=d$ and set 
$$
\zeta_t=c\Big(\int_{\bR^{d}}
|u_t(x)|^{p }\,dx\Big)^{1/(p -d )} 
\quad\text{with a constant $c$}, 
$$
and define 
$u^{M}_t(x)=u_t(x){\bf1}_{|u_t(x)|\geq \zeta_t}$,
$u^B_t=u_t-u^M_t=u_t{\bf1}_{|u_t|\leq \zeta_t}$. 
Then for any $B\in\bB_{\rho}$ and $\varepsilon>0$ 
$$
\dashint_{B }|u^{M}_t(x)|^{d}\,dx
\leq \zeta_t^{d-p }
\dashint_{B }|u_t(x)|^{p }\,dx
\leq N(d)c^{d-p }\rho^{-d}\leq \varepsilon \rho^{-d}
$$
for sufficiently large $c=c(d,p,\varepsilon)$, 
and clearly, 
$|u_t^{B} |\leq \zeta_t $ such that  
$$
\int_{0}^{T}\zeta_t^{2}\,dt
=c^{2}\int_{0}^{T}|u_t|^{2p/(p-q) }_{L_p}
\,dt<\infty,\,\, \text{since $2p/(p-q)\leq q$.}
$$
 If $p=\infty$ then we take $u^M=0$, $u^B=u$, 
$\hat u=0$ and $\bar u_t:=|u_t|_{L_{\infty}}$,  and notice that 
$$
\int_0^T\bar u_t^2\,dt=\int_0^T|u_t|^2_{L_{\infty}}\,dt<\infty\,\,\text{(a.s.), 
\,\,since $q\geq 2$},
$$
which completes the proof of the corollary. 
\end{proof}
\begin{remark}
                                                                           \label{remark1.14.3}
 If $u\in L_{d,\infty}$ (a.s.) 
 then by the above calculations $u$ is still 
admissible for $r=d$, but to be able to apply the corollary we 
need further conditions ensuring  
that the Morrey norm of $u^M$ is sufficiently small, uniformly 
in $(\omega,t)$.  We remark that in the case $d=3$  results 
on uniqueness and smoothness 
for Hopf-Leray week solutions to 
deterministic Navier-Stokes equations \eqref{eq0}-\eqref{ini0} 
(with $\frf=0$, $f=0$) are obtained in \cite{ESS2003} when 
the solutions belong to $L_{d,\infty}$, and 
for any $d\geq3$ such results are proved in \cite{DoDa2009}.

\end{remark}

\begin{remark}
                                                                             \label{remark 2.14.3}
The case $(d/p)+(2/q)<1$ in \eqref{LPS} is often 
 called the subcritical case and $(d/p)+(2/q)=1$ is called the critical case. 
In connection with this we note that there are admissible functions $u$ 
such that $r<d$  and $u^{M}(t,\cdot)\not\in L_{r+\varepsilon,\loc}$  
no matter how small $\varepsilon>0$ is. Thus Theorem \ref{theorem 2} 
covers examples also in the ``supercritical'' case.                                           
\end{remark}

\mysection{Preliminaries}
                                       \label{section prelim}

First we present a version of the It\^o formula, 
Theorem 3.1 from \cite{GK2024}. 
To this end let 
 $(V,(\cdot,\cdot)_{V})$ and 
$(H,(\cdot,\cdot)_{H})$ be some 
separable Hilbert spaces such that $V$ continuously 
and densely embedded in $H$. 

Assume we are given $V $-valued
functions $v,v^{\ast}$, an $H$-valued
function $F$, and an $\ell_2(H)$-valued function 
$G=(G^k)$, which are predictable functions 
on $\Omega\times[0,\infty)$ such that 
\begin{equation}
                                                   \label{8.9.5}
\int_{0}^{T}  | v_{t}|^2_{V}+|v^{*}_{t}  | ^{2}_{V }+|G_t|^2_{\ell_2(H)}\,dt
+\int_{0}^{T}  | F_{t}  |  _{H }\,dt<\infty 
\quad (a.s.) 
\end{equation}
for any $T\in(0,\infty)$.  
Let $v_{0}$ be an $H$-valued $\cF_0$-measurable 
random variable, and assumed that for every 
$v\in V$
we have
\begin{equation}
                                                 \label{4.9.7}
( v ,v_t)_{H }=( v ,v_{0})_{H }+\int_{0}^{t}
[( v ,v^{*}_s)_{V }+( v ,F_s)_{H}]\,ds
+\int_{0}^{t}(v,G^k_s)_{H}\,dw^k_s
\end{equation}
almost surely for $dt$-almost all $t\in(0,\infty)$. 
Then the following theorem holds. 

\begin{theorem}
                                  \label{theorem ito}
Under the above assumptions there exists a continuous
$H$-valued $\cF_t$-adapted process $(u_t)_{t\geq0}$
such that 

(i) $u_{t}=v_{t}$ for $P\otimes dt$-almost everywhere;

(ii) with probability one we have 
\begin{equation} 
                                                 \label{4.9.8}
( v ,u_{t})_{H }=( v ,v_{0})_{H }+\int_{0}^{t}
[( v ,v^{*}_{s})_{V }+( v ,f_{s})_{H}]\,ds
+\int_0^t(v,g_s^k)_{H}\,dw^k_s
\end{equation}
for all $v\in V$ and $t\geq0$;

(iii) with probability one 
\begin{align}
| u_{t}  | ^{2}_{H }= & | v_{0}  | ^{2}_{H }+
2\int_{0}^{t}
[(u_{s},v^{*}_{s})_{V }+(u_{s},F_{s})_{H}+\tfrac{1}{2}|G_s|^2_{\ell_2(H)}]
\,ds                                                                                        \nonumber\\
&+2\int_0^t(u_s,G^k_s)_{H}\,dw^k_s
\quad\text{for all $t\geq0$.}                                       \label{ito}
\end{align}
\end{theorem}

  \begin{lemma}
                                         \label{lemma 12.17.1}
If $f$ is real-valued admissible function  
then 
\begin{equation}
                                             \label{12.17.1}
\int_{\bR^{d}}|f_{t}|^{2}|\varphi|^{2}\,dx
\leq N\hat f^{2}\int_{\bR^{d}}|D\varphi|^{2}\,dx
+\int_{\bR^{d}}\big(N\rho_{0}^{-2}\hat f^{2}+2  
\bar f^{2}_{t} \big)|\varphi|^{2}\,dx 
\end{equation}    
for all $\varphi\in C_0^{\infty}$ with a constant $N=N(d,r)$. 
\end{lemma} 

This lemma follows from Lemma 3.5 of \cite{Kr_23}.   

Define  
\begin{equation}
 \frb(u,v,w):=\int_{\bR^d}u^j(x)D_jv^i(x)w^i(x)\,dx
\end{equation}
 for $\bR^d$-valued Borel functions $u$, $v$ and $w$ on $\bR^d$ when 
 the generalised derivatives $D_jv^i$ are functions on $\bR^d$,  
 and the integral is well-defined. 

\begin{corollary}
                                                              \label{corollary b}
(i) Let $u$ be a $\bR^d$-valued admissible function 
and let  $v, w\in\bW^1_2$ 
Then 
we have 
$$
\frb(u_t,v,w)
\leq N\hat u(|Dv|_{\bL_2}
+\rho_0^{-1}|v|_{\bL_2})|Dw|_{\bL_2}   
$$
$$
+N\bar u_t|Dv|_{\bL_2}|w|_{\bL_2},
$$
(ii) Let $w$ be a $\bR^d$-valued admissible function 
and let  $v, u\in\bW^1_2$. Then  
$$
\frb(u_t,v,w)\leq N\hat w(|Dv|_{\bL_2}
+\rho_0^{-1}|v|_{\bL_2})|Du|_{\bL_2}   
$$
\begin{equation*}
                                                         \label{w}
 +N\bar w_t|Dv|_{\bL_2}|u|_{\bL_2},                                                        
\end{equation*}
where $N$ is a constant depending only on 
$d$ and $r$. 
\end{corollary}
\begin{proof}
It is easy to see that Lemma \ref{lemma 12.17.1} holds also 
for all $\varphi\in W^1_2$. 
Hence using first the Cauchy-Schwarz-Bunyakovsky
inequality for the integral of the product of functions 
$u^{j}_tw^i$ and $D_jv^i$ and then applying 
\eqref{12.17.1} to $|u^{j}_t|^2|w^i|^2$ for each $i,j$,  
we obtain this corollary. We get \eqref{w} in the same way. 
\end{proof}

  \begin{lemma}
                                                                      \label{lemma antisymmetry}
Let $u$ be an $\bR^d$-valued admissible function and let 
$v,w\in\bW^1_2$. 
Then for all $(\omega,t)$ 
we have 
$|\frb(u_t,v,w)|<\infty$. 
Moreover, if $u_t\in\cV$ then 
$$
\frb(u_t,v,w)=-\frb(u_t,w,v).  
$$
\end{lemma}
\begin{proof}
Clearly, $|\frb(u_t,v,w)|<\infty$ 
by Corollary \ref{corollary b}. For $\varepsilon>0$ let 
$\varphi_{\varepsilon}\in C_0^{\infty}(\bR^d,\bR^d)$ such that 
$|v_t-\varphi_{\varepsilon}|_{\bW^1_2}\leq\varepsilon$. Then 
by integration by parts, by the Leibniz rule 
and using that ${\rm div}\, u_t=0$ we get 
\begin{equation}
                                                                 \label{1.12.2}
\int_{\bR^d}u^j_tD_j\varphi^i_{\varepsilon}w^i\,dx
=-\int_{\bR^d}\varphi_{\varepsilon}^iD_j(w^iu_t^j)\,dx 
=-\int_{\bR^d}u^j_tD_jw^i\varphi_{\varepsilon}^i\,dx. 
\end{equation}
By Corollary \ref{corollary b} we have 
$$
|\frb(u_t,v-\varphi_{\varepsilon},w)|\leq N\hat u(|Dv-D\varphi_{\varepsilon}|_{\bL_2}
+\rho_0^{-1}|v-\varphi_{\varepsilon}|_{\bL_2})|Dw|_{\bL_2}
$$
$$ 
+N\bar u_t|Dv-D\varphi_{\varepsilon}|_{\bL_2}|w|_{\bL_2},  
$$
which shows that 
$$
\frb(u_t,\varphi_{\varepsilon},w)\to \frb(u_t,v,w)\quad \text{for $\varepsilon\to0$}.
$$
In the same way we get 
$\lim_{\varepsilon\to0}\frb(u_t,w, \varphi_{\varepsilon})=\frb(u_t,w, v)$. 
Thus letting $\varepsilon\to0$ in \eqref{1.12.2} we finish the proof 
of the lemma. 
\end{proof}
\begin{corollary}
                                                                        \label{corollary 0}
Let $u$ be an $\bR^d$-valued admissible function such that $u_t\in\cV$  
for all $(\omega,t)$. Then 
$\frb(u_t,v,v)=0$ for $v\in \bW^1_2$, $(\omega,t)\in\Omega\times[0,T]$. 
\end{corollary}
\begin{corollary}
                                                                        \label{corollary v}
Let $u$ and $v$ be $\bR^d$-valued admissible function such that $u_t\in\cV$,  
and $v,w\in \bW^{1}_{2}$ for all $(\omega,t)$. Then 
$$
\frb(u_t,v,w)\leq N\hat v(|Dw|_{\bL_2}
+\rho_0^{-1}|w|_{\bL_2})|Du|_{\bL_2}   
$$
\begin{equation*}
                                                         \label{v}
 +N\bar v_t|Dw|_{\bL_2}|u_t|_{\bL_2}                                                         
\end{equation*}
with a constant $N=N(d,r)$.
\end{corollary}
\begin{proof}
We use that $\frb(u_t,v,w)=-\frb(u_t,w,v)$ and apply 
Corollary \ref{corollary b}.
\end{proof}

\mysection{Proof of Theorems \ref{theorem 1} and \ref{theorem 2}}

\begin{proof}[Proof of Theorem \ref{theorem 1}]
Let $u$ be an admissible solution on $[0,T]$. 
We set $H=\cH$, $V=\cV$, denote 
the conjugate of $V$ by $V^{*}$ and                                  
the duality between $v\in V$ and $v^{*}\in V^{*}$
by $\langle v, v^{*}\rangle$. Note that $v^{\ast}\in V^{\ast}$ 
can be identified with $Qv^{\ast}\in V$ such that 
$\langle v, v^{\ast}\rangle=(v,Qv^{\ast})_V$ holds for every $v\in V$, 
and thus $|v^{\ast}|_{V^{\ast}}=|Qv|_V$. 

We will apply Theorem \ref{theorem ito} to prove 
Theorem \ref{theorem 1}. 
By definition of the solution, 
\begin{equation}  
                                                                            \label{1.8.10.24}
\int_0^T|u(t)|^2_V\,dt<\infty\quad (a.s.).
\end{equation}
We are going to define $v^{\ast}_t\in V^{\ast}$, $F_t\in H$ and 
$G_t=(G_t^k)\in \ell_2(H)$  by requiring 
\begin{equation}
                                                                                        \label{v*}
\langle v,v^{\ast}_t\rangle
=-(a^{ij}_tD_ju_t+\frf^i_t(u_t), D_iv)-(u^{M j}D_ju_t,v),
\end{equation}
\begin{equation}
                                                                                             \label{F}
(v, F_t)_H=
(-u_t^{Bj}D_ju+f_t(u,Du)+\gamma_t^{kj}D_jq^k_t,v)
\end{equation}
and 
\begin{equation}
                                                                                               \label{G}
(v,G^k_t)_H=(\sigma_t^{jk}D_ju_t+h^k_t(u),v) 
\quad \text{for $k\geq1$}, 
\end{equation}
respectively for all $v\in\cV$.
By Assumptions \ref{assumption 1}, \ref{assumption 2} (i) and 
Corollary \ref{corollary b} 
we have 
$$
|(a^{ij}_tD_ju_t+\frf^i_t(u_t), D_iv)|
\leq (\delta^{-1}|Du|_{L_2}+\lambda_t|u_t|_{\bL_2}+|\frF_t|_{L_2})|v|_{V}, 
$$
$$
|(u^{M j}D_ju_t,v)|\leq N\hat u(|Du|_{\bL_2}
+\rho_0^{-1}|u|_{\bL_2})|v|_{V}.  
$$
Hence we can see, noticing that 
\begin{equation}                                      
                                                                      \label{esssup 1}
\int_0^T\lambda^2_t|u_t|^2_{\bL_2}\,dt
\leq\esssup_{t\in[0,T]}|u_t|^2_{\bL_2}\int_0^T\lambda^2_t\,dt<\infty, 
\end{equation}
that there is a predictable $V^{\ast}$-valued process $v^{\ast}$, 
identified with the $V$-valued process $Qv^{\ast}$  such that 
\begin{equation}
                                                                      \label{1.15.2}
\int_0^T|v^{\ast}_t|_{V^{\ast}}^2\,dt
=\int_0^T|Qv^{\ast}_t|_{V}^2\,dt<\infty\quad (a.s.) 
\end{equation}
and \eqref{v*} holds. 
Using the definition of admissible functions  
and Assumption \eqref{assumption 1}\,(i)
we have 
$$
|-u_t^{Bj}D_ju_t+f_t(u_t,Du)|_{\bL_2}
\leq \bar u_t|Du_t|_{L_2}+\kappa_t|u_t|_{L_2}
+\chi_t|Du_t|_{\bL_2}+|\frG_t|_{L_2},  
$$
where 
\begin{equation}                                             \label{esssup 2}
\int_0^T\kappa_t|u_t|_{\bL_2}\,dt
\leq\esssup_{t\in[0,T]}|u_t|_{\bL_2}\int_0^T\kappa_t\,dt<\infty
\,\, (a.s.), 
\end{equation}
and by the Cauchy-Bunyakovsky-Schwarz inequality  
$$
\int_0^T(\bar u_t|Du_t|_{\bL_2}
+\chi_t|Du_t|_{\bL_2})\,dt<\infty.
$$
Using \eqref{q}, Assumptions \ref{assumption 1} and 
\ref{assumption 2} (i)
$$
|\gamma_t^{ki}D_iq_t^k|_{\bL_2}
=|\gamma_t^{ki}(\cR M^k_t(u_t))^i|_{\bL_2}
\leq
\vartheta_t\Big(\sum_{k=1}^{\infty}|\cR M_t^k(u_t))|^2_{\bL_2}\Big)^{1/2}
$$
$$
\leq \vartheta_t\Big(\sum_{k=1}^{\infty}|M_t^k(u_t)|^2_{\bL_2}\Big)^{1/2}
\leq 
N(\vartheta_t|Du_t|_{\bL_2}+\vartheta_t\lambda_t|u_t|_{\bL_2}+\vartheta_t|\frF_t|_{L_2})
$$
with a constant $N=N(K,d)$. Notice that 
\begin{equation}
                                                                            \label{esssup 3}
\int_0^T\vartheta_t\lambda_t|u_t|_{\bL_2}\,dt
\leq 
\esssup_{t\in[0,T]}|u_t|_{\bL_2}\int_0^T(\vartheta^2_t
+\lambda^2_t)\,dt<\infty 
\,\,(a.s.).
\end{equation}
Hence 
$F_t:=\cS\big(-u_t^{Bj}D_ju_t+f_t(u_t,Du)+\gamma_t^{ki}D_iq_t^k\big)$, 
$t\in[0,T]$,  
is an $H$-valued predictable process such that \eqref{F} 
holds and 
\begin{equation}
                                                                                     \label{2.15.2}
\int_0^T|F_t|_{H}\,dt
\leq 
\int_0^T|-u_t^{Bj}(D_ju_t+f_t(u_t,Du)
+\gamma_t^{ki}D_iq_t^k|_{\bL_2}\,dt
<\infty\,\,(a.s.).
\end{equation}
Notice that by Assumptions \ref{assumption 1} (i), 
\ref{assumption 2} (i) and Remark \ref{remark sigma} 
we have 
$$
\sum_{k}|\sigma^{ik}D_iu_t+h^{k}_t(u_t)|^2_{\bL_2}
\leq N(|Du_t|^2_{\bL_2}+\lambda^2_t|u_t|^2_{\bL_2}) 
$$
with $N=N(\delta,d)$, and 
by the definition of a solution and by \eqref{1.15.2} 
$$
\int_0^T|Du_t|^2_{\bL_2}+\lambda^2_t|u_t|^2_{\bL_2}
\leq \int_0^T|Du_t|^2_{\bL_2}\,dt
+
\esssup_{t\in[0,T]}|u_t|^2_{\bL_2}\int_0^T\lambda^2_t\,dt<\infty 
\,(a.s.). 
$$
Hence, because the operator norm of $\cS$ is 1,  
$$
G_t=(G_t^k)_{k=1}^{\infty}
:=(\cS(\sigma^{ik}D_iu_t+h^{k}_t(u_t)))_{k=1}^{\infty},
\quad t\in[0,T], 
$$
is an $\ell_2(H)$-valued predictable process 
such that 
\begin{equation}
                                                                             \label{3.15.2}
\int_0^T|G_t|^2_{\ell_2(H)}\,dt<\infty \,\,(a.s.),  
\end{equation}
and \eqref{G} holds. For $v\in \cV$ the convolution 
$v^{(\varepsilon)}=v\ast k_{\varepsilon}$ 
belongs to $\cV\cap L_d$ for $\varepsilon>0$, 
and $\lim_{\varepsilon\to0}|v-v^{\varepsilon}|_{\cV}=0$,  
when 
$k_{\varepsilon}=\varepsilon^{-d}k(\cdot/\varepsilon)$. 
and $k\in C_0^{\infty}(\bR,\bR^d)$ has compact support 
and unit integral. 
This shows that $\cV\cap L_d$ is dense in $\cV$. 
Hence, due to  
\eqref{v*}, \eqref{F}, \eqref{G}, \eqref{1.15.2}, \eqref{2.15.2} 
and the definition of the solution, we have 
\eqref{8.9.5} and \eqref{4.9.7}. Thus we can apply Theorem 
\ref{theorem ito} to get that a $P\otimes dt$-modification 
of $u$, denoted also by $u$, is an $H$-valued 
continuous process and 
almost surely 
\begin{align}
| u_{t}  | ^{2}_{H }= | u_{0}  | ^{2}_{H }
-2&\int_{0}^{t}\big[
\big(a^{ij}_{s}D_iu_s+\frf_s(u_s),D_ju_s\big)
+\big(u_{s(u_s)},u_s\big)\big]\,ds                                                      \nonumber\\
+2&
\int_{0}^{t}
\big(f_s(u_s,Du_s)+\gamma_s^{kj}D_ju_s,u_s\big)\,ds               \nonumber\\                        
+&\int_{0}^{t}
\sum_{k}\big|\cS\big(\sigma_s^{ik}D_iu_s+h_s(u_s)\big)\big|^2_{H}\,ds
+m_t                                                                                      \label{4.15.2}
\end{align}
holds for all $t\in[0,T]$, 
where 
$$                                                                                              
m_t=2\int_0^t
\big(\sigma^{ik}D_iu_s+h_s^k(u_s),u_s\big)\,dw^k_s. 
$$

To show the estimate \eqref{u estimate} 
we may assume that the expression 
on right-hand side of \eqref{u estimate} is finite. First note that 
by Corollary \ref{corollary 0} we have 
\begin{equation}
                                                                         \label{b0}
(u_{t(u_t)},u_t)=0 
\quad
\text{($P\otimes dt$-a.e)}.
\end{equation}
Then use that Assumptions \ref{assumption 1} and \ref{assumption 2}\,(i) to get  
$$
I_s:=-2\big(a^{ij}_{s}D_iu_s-2\frf_s(u_s),D_ju_s\big)
-2\big(f_s(u_s,Du_s)+2\gamma_s^{kj}D_ju_s,u_s\big) 
$$ 
$$         
\leq-2\big(a^{ij}_{s}D_iu_s,D_ju_s\big)
+2\lambda_s|u_s|_{\bL_2}|Du_s|_{L_2}+2|\frF_s|_{L_2}|Du_s|_{L_2}
$$
$$
+2\kappa_s|u_s|^2_{\bL_2}+2\chi_s|Du_s|_{L_2}|u_s|_{\bL_2}
$$
$$
+2|\frG_s|_{L_2}|u_s|_{\bL_2}
+2\vartheta_s|Du_s|_{L_2}|u_s|_{\bL_2}
$$
$$
\leq -2\big(a^{ij}_{s}D_iu_s,D_ju_s\big)+\tfrac{\delta}{2}|Du_s|^2_{L_2}
+N(\lambda^2_s+\chi^2_s+\vartheta_s^2+\kappa_s)|u_s|^2_{\bL_2}
$$
\begin{equation}
                                                                                 \label{I}
+|\frF_s|^2_{L_2}+2|\frG_s|_{L_2}|u_s|_{\bL_2}, 
\end{equation}
and 
$$
J_s:=\sum_{k}|\cS\big(\sigma_s^{ik}D_iu_s+h^k_s(u_s)\big)\big|^2_{H}
\leq \sum_{k}\big|\sigma_s^{ik}D_iu_s+h^k_s(u_s)\big|^2_{\bL_2}
$$
$$
\leq \big(\sigma_s^{ik}\sigma_s^{ik}D_iu_s,D_ju_s\big)
+2\big|\big(\sigma_s^{ik}D_iu_s,h^k_s(u_s)\big)\big|
+\sum_{k}\big|h^k_s(u_s)\big|^2_{\bL_2}
$$
\begin{equation}
                                                                                  \label{J}
\leq \big(\sigma_s^{ik}\sigma_s^{ik}D_iu_s,D_ju_s\big)
+\tfrac{\delta}{2}|Du_s|^2_{L_2}+N\lambda_s^2|u_s|^2_{\bL_2}
+N
|\frF_s|^2_{L_2}
\end{equation}
with a constant $N=N(d,\delta)$. 
Note that by Assumption \ref{assumption 1} 
we have 
$$
\big((2a^{ij}_s-\sigma^{ik}\sigma^{jk})D_iu_s,D_ju_s\big)
\geq 2\delta|u_s|^2_{\bL_2}. 
$$
Thus, by \eqref{I} and \eqref{J}  we get 
$$
I_s+J_s\leq -\delta|Du_s|^2_{L_2}
+N(\lambda_s^2+\chi^2_s+\vartheta_s^2+\kappa_s)|u_s|^2_{\bL_2}
$$
$$
+N|\frF_s|^2_{L_2}+2|\frG_s|_{L_2}|u_s|_{\bL_2}, 
$$
Using this and \eqref{b0}, from \eqref{4.15.2} we obtain 
$$
d|u_t|^2_{\bL_2}\leq \big(-\delta|Du_t|^2_{L_2}
+N(\lambda_t^2+\chi^2_t+\vartheta_t^2+\kappa_t)|u_t|^2_{\bL_2}\big)\,dt
$$
\begin{equation}
                                                                \label{1.16.2}
(N|\frF_t|^2_{L_2}+2|\frG_t|_{L_2}|u_t|_{\bL_2})\,dt+dm_t
\end{equation}
with a constant $N=N(\delta)$.
Hence 
$$
d|e^{-\varphi_t}u_t|^2_{\bL_2}\leq \big(-\delta |e^{-\varphi_t}Du_t|^2_{L_2}
-\alpha_t|e^{-\varphi_t}u_t|^2_{\bL_2}\big)\,dt
$$
$$
+\big(
N|e^{-\varphi_t}\frF_t|^2_{L_2}+2e^{-2\varphi}|u_t|_{\bL_2}|\frG_t|_{L_2}
\big)\,dt+e^{-2\varphi_t}\,dm_t.  
$$
From this, by standard stopping time argument and the using H\"older's and Young's 
inequalities  we get 
$$
E|e^{-\varphi_{T\wedge\tau}}u_{T\wedge\tau}|^2_{\bL_2}
+\delta E\int_0^{T\wedge\tau}|e^{-\varphi_s}Du_s|_{L^2}^2\,ds
+E\int_0^{T\wedge\tau}\alpha_s|e^{-\varphi_s}u_s|^2_{\bL_2}\,ds
$$
\begin{equation}
                                                                                         \label{2.16.2}                                                                                       
\leq E |u_0|^2_{\bL_2}
+NE\int_0^{T\wedge\tau}|e^{-\varphi_s}\frF_s|^2_{L_2}\,ds
+2E\int_0^{T\wedge\tau}e^{-2\varphi_s}|u_s|_{\bL_2}|\frG_s|_{L_2}\,ds
\end{equation}                                                                                 
and 
$$
E\sup_{t\leq T\wedge\tau}|e^{-\varphi_{t}}u_{t}|^2_{\bL_2}
\leq E|u_0|^2_{\bL_2}
+E\int_0^{T\wedge\tau}|e^{-\varphi_s}\frF_s|^2_{L_2}\,ds
$$
$$
+2E\int_0^{T\wedge\tau}e^{-2\varphi_s}|u_s|_{\bL_2}|\frG_s|_{L_2}\,ds
+E\sup_{t\leq T\wedge\tau}\int_0^{t}e^{-2\varphi_s}\,dm_s
$$
$$
\leq E|u_0|^2_{\bL_2}
+E\int_0^{T\wedge\tau}|e^{-\varphi_s}\frF_s|^2_{L_2}\,ds
+\tfrac{1}{4}E\sup_{t\leq T\wedge\tau}|e^{-\varphi_t}u_t|^2_{\bL_2}
$$
\begin{equation}
                                                                                           \label{3.16.2}
+4\Big(E\int_0^{T\wedge\tau}e^{-\varphi_s}|\frG_s|_{L_2}\,ds\Big)^2
+E\sup_{t\leq T\wedge\tau}\int_0^{t}e^{-2\varphi_s}\,dm_s
 \end{equation}                                                                                         
for any stopping time $\tau\leq T$. 
By the Davis inequality, 
H\"older's and Young's inequalities we have  
$$
E\sup_{t\leq T\wedge\tau}\int_0^{t}e^{-2\varphi_s}\,dm_s
\leq 6E
\Big(\int_0^{T\wedge\tau}e^{-4\varphi_s}
J_s |u_s|^2_{\bL_2}\,ds
\Big)^{1/2}
$$
\begin{equation}
                                                       \label{4.16.2}
\leq 
\tfrac{1}{8}
E\sup_{t\leq T\wedge\tau}|e^{-\varphi_s}u_s|^2_{\bL_2}+
NE\int_0^{T\wedge\tau}e^{-2\varphi}J_s\,ds.  
\end{equation}
By \eqref{J} and using \eqref{2.16.2},  for the last term here we get 
$$
NE\int_0^{T\wedge\tau}e^{-2\varphi}J_s\,ds
$$
$$
\leq N_1\int_0^{T\wedge\tau}\big(|e^{-\varphi_s}Du_s|^2_{L_2}+
\lambda_s^2|e^{-\varphi_s}u_s|^2_{\bL_2}
+|e^{-\varphi_s}\frF_s|^2_{\bL_2}\big)\,ds
$$
$$
\leq N_1E |u_0|^2_{\bL_2}
+N_2E\int_0^{t\wedge\tau}|e^{-\varphi_s}\frF_s|^2_{L_2}\,ds
+N_2E\int_0^{t\wedge\tau}e^{-2\varphi_s}|u_s|_{\bL_2}|\frG_s|_{L_2}\,ds
$$
$$
\leq N_1E |u_0|^2_{\bL_2}
+N_2E\int_0^{t\wedge\tau}|e^{-\varphi_s}\frF_s|^2_{L_2}\,ds
$$
\begin{equation*}                                                                        
+\tfrac{1}{16}
E\sup_{t\leq T\wedge\tau}|e^{-\varphi_s}u_s|^2_{\bL_2}
+N_3E\Big(\int_0^{t\wedge\tau}e^{-\varphi_s}|\frG_s|_{L_2}\,ds\Big)^{2}
\end{equation*}
with constants $N_1$, $N_2$ and $N_3$ depending only 
on $d$ and $\delta$. 
Consequently, combining this with \eqref{4.16.2}, 
from \eqref{3.16.2} we get 
$$
E\sup_{t\leq T\wedge\tau}|e^{-\varphi_{t}}u_{t}|^2_{\bL_2}
\leq NE|u_0|^2_{\bL_2}
+NE\int_0^{T\wedge\tau}|e^{-\varphi_s}\frF_s|^2_{L_2}\,ds
$$
$$
+NE\Big(\int_0^{T\wedge\tau}e^{-\varphi_s}|\frG_s|_{L_2}\,ds\Big)^{2}
+\tfrac{1}{2}E\sup_{t\leq T\wedge\tau}|e^{-\varphi_t}u_t|^2_{\bL_2}
$$                                                                                    
for every stopping time $\tau\leq T$. Taking here in place of $\tau$ the 
stopping times 
$$
\tau_{n}=\inf\{t\in[0,T]: |u_t|_{\bL_2}\geq n\}\wedge\tau
$$
for integers $n\geq1$, we get 
$$
E\sup_{t\leq T\wedge\tau_n}|e^{-\varphi_{t}}u_{t}|^2_{\bL_2}
\leq 2NE|u_0|^2_{\bL_2}
+2NE\int_0^{T\wedge\tau}|e^{-\varphi_s}\frF_s|^2_{L_2}\,ds
$$
$$
+2NE\Big(\int_0^{T\wedge\tau}e^{-\varphi_s}|\frG_s|_{L_2}\,ds\Big)^{2}.  
$$
Letting here $n\to\infty$ by Fatou's lemma  we obtain 
$$
E\sup_{t\leq T\wedge\tau}|e^{-\varphi_{t}}u_{t}|^2_{\bL_2}
\leq 2NE|u_0|^2_{\bL_2}
+2NE\int_0^{T\wedge\tau}|e^{-\varphi_s}\frF_s|^2_{L_2}\,ds
$$
\begin{equation}
                                                                                          \label{2.17.2}
+2NE\Big(\int_0^{T\wedge\tau}e^{-\varphi_s}|\frG_s|_{L_2}\,ds\Big)^{2}.  
\end{equation}
We use this to estimate the last term in \eqref{2.16.2} as
$$ 
2E\int_0^{T\wedge\tau}e^{-2\varphi_s}|u_s|_{\bL_2}|\frG_s|_{L_2}\,ds
$$
$$
\leq E\sup_{t\leq T\wedge\tau}|e^{-\varphi_{t}}u_{t}|^2_{\bL_2}
+E\Big(\int_0^{T\wedge\tau}|e^{-\varphi_s}\frG_s|_{L_2}\,ds\Big)^2
$$
$$
\leq NE|u_0|^2_{\bL_2}
+NE\int_0^{T\wedge\tau_n}|e^{-\varphi_s}\frF_s|^2_{L_2}\,ds
+NE\Big(\int_0^{T\wedge\tau_n}e^{-\varphi_s}|\frG_s|_{L_2}\,ds\Big)^{2}.    
$$
Thus from \eqref{2.16.2} we get 
$$
E|e^{-\varphi_{T\wedge\tau}}u_{t\wedge\tau}|^2_{\bL_2}
+E\int_0^{T\wedge\tau}|e^{-\varphi_s}Du_s|_{\bL^2}^2\,ds
+E\int_0^{T\wedge\tau}\alpha_s|e^{-\varphi_s}u_s|^2_{\bL_2}\,ds
$$
\begin{equation*}                                                                                                                                                                       
\leq NE |u_0|^2_{\bL_2}
+NE\int_0^{T\wedge\tau}|e^{-\varphi_s}\frF_s|^2_{L_2}\,ds
+2E\Big(\int_0^{T\wedge\tau}|e^{-\varphi_s}\frG_s|_{L_2}\,ds\Big)^2
\end{equation*}                                                                      
with a constant $N=N(d,\delta)$. Combining this with 
\eqref{2.17.2} we obtain the estimate \eqref{u estimate} with any 
stopping time $\tau\leq T$ in place of $T$.

\end{proof}

\begin{proof}[Proof of Theorem \ref{theorem 2}] 
For $\fru=u-v$ we have that for each $\varphi\in\cV\cap\bL^d$
$$
(\fru_t,\varphi)
=-\int_0^t\big[(a^{ij}_sD_j\fru_s+\frf_s^{i}(u_s)-\frf^{i}_s(v_s),D_i\varphi)
+(u_{s(u_s)}-v_{s(v_s)}, \varphi)\big]\,ds
$$
$$
+\int_0^t
(f_s(u_s,Du_s)-f_s(v_s,Dv_s),\varphi)\,ds
$$
$$
+\int_0^t
\Big(\gamma_s^{ki}\big(\cR\big(M^k_s(u_s)-M^k_s(v_s)\big)\big)^i,\varphi\Big)
\,ds
$$
\begin{equation}
                                                                               \label{3.17.2}
+\int_0^t(\sigma_s^{ik}D_i\fru_s+h_s^k(u_s)-h_s^k(v_s),\varphi)\,dw^k_s
\end{equation}
holds for $P\otimes dt$-a.e. $(\omega,t)\in[0,T]$, where $M_s^k(\cdot)$ 
is defined in \eqref{M}.  As in the proof of 
Theorem \ref{theorem 1}, we set $V:=\cV$, $H:=\cH$, use the notation 
$V^{\ast}$ and $H^{\ast}$ for their conjugate spaces, 
$\langle\cdot,\cdot\rangle$ for the duality product between $V$ and $V^{\ast}$,  
and will use Theorem \ref{theorem ito}. 
In the same fashion as a similar statement in the proof of Theorem \ref{theorem ito} 
is established, we can show 
the existence of a $V^{\ast}$-valued predictable process $v^{\ast}$ 
such that 
$$
\int_0^T|v^{\ast}|^2_{V^{\ast}}\,ds<\infty \,\,(a.s.), 
$$
 and for all $\varphi\in V$
$$
\langle \varphi,v^{\ast}_s\rangle
=-(a^{ij}D_j\fru_s+\frf_s^{i}(u_s)-\frf^{i}_s(v_s),D_i\varphi)-
(u_s^{Mi}D_iu_s-v_s^{Mi}D_iv_s,\varphi)
$$
holds for all $(\omega,t)$. Moreover, 
we can verify that 
$$
F_s:=\cS
\big(v_{s}^{Bi}D_iv_{s}-u_{s}^{Bi}D_iu_{s}
+f_s(u_s,Du_s)-f_s(v_s,Dv_s)\big)
$$
$$
+\cS\Big(\gamma_s^{ki}\big(\cR\big(M^k_s(u_s)-M^k_s(v_s))\big)^i\Big), 
\quad t\in[0,T] 
$$
and $G_t=(G^k_t)_{k=1}^{\infty}$, defined by
$$
G^k_t:=\cS\Big(\sigma^{ik}D_i\fru_s+h^{k}_s(u_s)-h^{k}_s(v_s)\Big)
$$
are $H$-valued and $\ell_2(H)$-valued predictable processes such that 
$$
\int_0^T\big[|G_t|^2_{\ell_2(H)}+|F_s|_{H}\big]\,ds<\infty\,\,(a.s.).
$$
Hence, because $\cV\cap \bL_d$ is dense in $V$, 
by virtue of \eqref{3.17.2} we get \eqref{4.9.7} (with $\fru_t$ in place of 
$v_t$) for all $v\in V$. Thus we can use Theorem \ref{theorem ito} 
to get that $\fru=(\fru_t)_{t\in[0,T]}$ has an $H$-valued  continuous 
$P\otimes dt$-modification,  
denoted also by $\fru$, such that almost surely 
\begin{align}
|\fru_t|^2_{\bL_2}=&-2\int_0^t\Big(a^{ij}D_j\fru_s
+\frf^i_s(u_s)-\frf^i_s(v_s),D_j\fru_s\Big)\,ds                 \nonumber\\
&-2\int_0^t\frb(u_s,u_s,\fru_s)-\frb(v_s,v_s,\fru_s)\,ds     \nonumber\\
&+2\int_0^t\big(f(u_s,Du_s)-f(v_s,Dv_s),\fru_s\big)\,ds  \nonumber\\
&+2\int_0^t
   \big(
   \gamma_s^{ki}\big(\cR(M^k_s(u_s)-M^k_s(v_s))\big)^i,\fru_s
   \big)\,ds                                                                     \nonumber\\
&+\int_0^t
    \sum_{k}\Big|\cS\big(\sigma_s^{ik}D_i\fru_s+h^k_s(u_s)-h^k_s(v_s)\big)
    \Big|^2_{H}\,ds+m_t                                                 \label{1.19.2}
\end{align}
for all $t\in[0,T]$, where 
$$
m_t=2\int_0^t
\big(\sigma_s^{ik}D_i\fru_s+h_s^{k}(u_s)-h^k_s(v_s), \fru_s \big)\,dw^k_s. 
$$
By Assumption \ref{assumption 2}\,(ii), using Young's inequality 
\begin{equation*}
-2\big(a^{ij}_sD_j\fru_s
+\frf^i_s(u_s)-\frf^i_s(v_s),D_j\fru_s\big)
\end{equation*}
\begin{equation}
                                                                            \label{2.19.2}
\leq 
-2\big(a^{ij}_sD_j\fru_s,D_i\fru_s\big)+\tfrac{\delta}{2}|D\fru_s|^2_{L_2}
+N_1\lambda^2_s|\fru_s|^2_{\bL_2}
\end{equation}
with a constant $N_1=N_1(\delta)$.    
To estimate  the absolute value of 
$$
I:=\frb(v_s,v_s,\fru_s)-\frb(u_s, u_s,\fru_s)
$$
note first that without loss of generality we may assume that $\hat u\leq \hat v$. 
Using the linearity in the first two 
arguments of the trilinear functional $\frb$, we have the identity
$$
 I=I_1+I_2+I_3, 
 $$
with $I_1=-\frb(u_s, \fru_s,\fru_s)$, $I_2=\frb(\fru_s,\fru_s,\fru_s)$, 
$I_3:=-\frb(\fru_s, u_s,\fru_s)$. 
Using Corollaries \ref{corollary 0} and \ref{corollary v} we have 
$I_1=I_2=0$ and 
$$
2|I|=2|I_3|\leq 
 N\hat u(|D\fru_s|_{L_2}
+\rho_0^{-1}|\fru|_{\bL_2})|D\fru_s|_{L_2}
 +N\bar u_s|D\fru_s|_{L_2}|\fru_s|_{\bL_2}                                                         
$$
\begin{equation}
                                                            \label{4.17.2}
 \leq N\hat u|D\fru_s|^2_{L_2}
 +\tfrac{\delta}{4}|D\fru_s|^2_{L_2}
 +N_1(\hat u^2+\bar u_s^2)|\fru_s|^2_{\bL_2}
\end{equation}                                                           
with constants $N=N(d,r)$ and $N_1=N_1(d,r,\delta,\rho_0)$. 
By \eqref{f}, using the Cauchy-Schwarz-Bunyakovsky 
and the Young inequalities we have 
$$
2\big(f(u_s,Du_s)-f(v_s,Dv_s),\fru_s\big)
\leq 2\big(
\kappa_s|\fru_s|_{\bL_2}+\chi_s|D\fru_s|_{L_2}
\big)|\fru_s|_{\bL_2}
$$
\begin{equation}
                                                                      \label{3.19.2}
\leq \tfrac{\delta}{8}|D\fru_s|^2_{L_2}
+N_2(\kappa_s+\chi^2_s)|\fru_s|^2_{\bL_2}
\end{equation}
with a constant $N_2=N_{2}(\delta)$. 
Moreover, 
$$
K_s:=\Big(
   \gamma_s^{ki}\big(\cR(M^k_s(u_s)-M^k_s(v_s))\big)^i,\fru_s
\Big)
$$
$$
\leq   \vartheta_s\Big|\big|\cR\big((M_s(u_s)-M_t(v_s))\big|_{\ell_2(\bR^d)}
|\fru_s|\Big|_{L_1}
$$
$$
 \vartheta_s\Big|\big|\cR\big((M_s(u_s))-M_s(v_s))\big|_{\ell_2(\bR^d)}\Big|_{L_2}
|\fru_s|_{\bL_2}. 
$$
Since the operator norm of $\cR$ is 1,  we get 
$$
\Big|\big|\cR\big((M_s(u_s))-M_s(v_s))\big|_{\ell_2(\bR^d)}\Big|^2_{L_2}
= 
\sum_{k=1}^{\infty}\big|\cR\big(M_s^k(u_s)-M_s^k(v_s)\big)\big|_{\bL_2}^2
$$
$$
\leq \sum_{k=1}^{\infty}\big|M_s^k(u_s)-M_s^k(v_s)\big|_{\bL_2}^2
=||M_s(u_s)-M_s(v_s)|_{\ell_2(\bR^d)}\big|_{L_2}^2
$$
$$
\leq 
\Big|\big||\sigma_s^iD_i\fru_s|_{\bR^d}\big|_{\ell_2}
+|h_s(u_s)-h_s(v_s)|_{\ell_2(\bR^d)}
\Big|^2_{L_2}. 
$$
Hence by Remark \ref{remark sigma} and condition \eqref{dh}, 
we obtain 
\begin{align}
K_s&\leq \vartheta_s N\big||D\fru_s|
+\lambda_s|\fru_s|\big|_{L_2}
|\fru_s|_{\bL_2}                                                    \nonumber\\                           
&\leq \tfrac{\delta}{16}|D\fru_s|^2_{L_2}
+N'(\vartheta^2_s+\lambda^2_s)|\fru_s|^2_{\bL_2}  \label{4.19.2}
\end{align}
with constants $N$ and $N'$ depending only on $d$ and $\delta$. 
Since the operator norm of $\cS$ is 1, by 
Assumption \ref{assumption 1} and condition \eqref{dh}, using 
the Cauchy-Schwarz-Bunyakovsky and the Young inequalities 
we have 
$$
\sum_{k} \Big|\cS\big(\sigma_s^{ik}D_i\fru_s+h^k_s(u_s)-h^k_s(v_s)\big)
    \Big|^2_{H}
$$
$$
\leq\sum_{k}\Big|\sigma_s^{ik}D_i\fru_s+h^k_s(u_s)-h^k_s(v_s)
    \Big|^2_{H}
$$
\begin{equation}
                                                                            \label{5.19.2}
\leq(\sigma_s^{ik}\sigma_s^{jk}D_i\fru_s,D_j\fru_s)+
\tfrac{\delta}{32}|D\fru_s|^2_{\bL_2}+N\lambda_s^2|\fru_s|^2_{\bL_2}
\end{equation}
with a constant $N=N(d,\delta)$. Consequently, by virtue of 
the estimates 
\eqref{2.19.2} through \eqref{5.19.2} and taking into account 
that 
$$
-\big((2a^{ij}-\sigma^{ik}\sigma^{jk})D_j\fru_s,D_i\fru_s\big)
\leq -2\delta|D\fru_s|^2_{\bL_2}, 
$$
from \eqref{1.19.2} 
we get 
$$
d|\fru_t|^2_{\bL_2}
\leq -(\delta-N\hat u)|D\fru_t|^2\,dt+\zeta_t|\fru_t|_{\bL_2}\,dt+dm_t, 
$$
where $N=N(d,r)$ is a constant, and 
$$
\zeta_t=N'(\hat u+\bar u_t^2+\lambda^2_t+\chi^2+\vartheta^2_t+\kappa_t)
$$
with a constant $N'=N'(d,r,\delta,\rho_0)$. Set 
$$
\phi_t:=\int_0^t\zeta_s\,ds, \quad t\in[0,T].
$$
Then for $\hat u\leq\delta/N$ we have 
$$
e^{-\phi_t}|\fru_t|^2_{\bL_2}\leq \int_0^te^{-\phi_s}dm_s=:m'_t,   
\quad t\in[0,T]. 
$$
Here $m'=(m'_t)_{t\in[0,T]}$ is a nonnegative local martingale starting from 0.  
Then almost surely $m'_t=0$ for all $t\in[0,T]$, which finishes the proof 
of the theorem. 
\end{proof}

\mysection{On $W^1_2$-regularity of admissible solutions}

Let $u$ be an $\cH$-solution to \eqref{eq}-\eqref{ini} in $[0,T]$, and consider the 
system of equations

\begin{align} 
                                                                      \label{eq2}
dv_t=&\big(D_i(a^{ij}_{t}D_{j}v_t+\frf_t^i(u_t))
+f_t(u_t,Du_t)
-v_{t(u_t)}-\nabla p_t+\gamma^{ki}_tD_{i}q_t^k\big)\,dt                                 \nonumber\\
&+(\sigma^{ik}_{t}D_{i}v_t+h_t^{k}(u_t)
-\nabla q^k_t)\,dw^k_t, 
\quad 
\text{div}\, v_t=0 
\end{align} 
on $\Omega\times[0,T]\times\bR^d$, with initial condition 
\begin{equation}                                       
                                                                 \label{ini2}
v_0(x)=u_0(x),\quad x\in\bR^d, \quad 
\end{equation}
for
$v=(v_t^{1}(t,x),...,v^{d}(t,x))$, 
where 
\begin{equation}
                                                                   \label{sol2b}
 \nabla q^k_t=\cR(\sigma_t^{jk}D_jv_t+h^k_t(u_t))
 \quad 
 \text{$P\otimes dt\otimes dx$-a.e. \,\, for every $k\geq1$}.                                                                                     
\end{equation}

Like the $\cH$ solution $u$ to \eqref{eq}-\eqref{ini} is defined, we say that 
$v$ is an $\cH$-solution to  \eqref{eq2}-\eqref{ini2} if it is an $\cH$-valued 
continuous $\cF_t$-adapted process, 
$$
\int_0^T|v_t|^2_{\cV}<\infty\,\,(a.s.),
$$
and almost surely 
\begin{align}
(v_t,\psi)=&(u_0,\psi)
 -\int_0^t\big[(a_s^{ij}D_jv_s+\frf^i_s(u_s),D_i\psi)
 +(v_{s(u_s)},\psi)\big]\,ds                                                      \nonumber\\
&+\int_0^t \big[(\gamma^{ki}_s\cR(\sigma_s^{jk}D_jv_s+h^k_s(u_s))^{i}
+f_s(u_s,Du_s), \psi)\big]\,ds      \nonumber\\
 &+\int_0^t(\sigma_s^{jk}D_jv_s+h^k_s(u_s),\psi)\,dw^k_s,                    \label{1.25.2}
 \end{align}
 for all $t\in[0,T]$ and $\psi\in \cV$.

Clearly,  $u$ is an $\cH$-solution to  \eqref{eq2}-\eqref{ini2}. 
We are going to use this to raise the regularity of $u$. 
To this end we make the following additional 
assumption. 
\begin{assumption}
                                                                             \label{assumption 4}
There exist constants $K_i$ (i=1,2,3), 
a nonnegative predictable process 
$\theta=\theta_t$ and a nonnegative 
$\cP\otimes\cB(\bR^d)$-measurable 
function $\frH=\frH_t(x)$ 
on $\Omega\times[0,T]\times\bR^d$, such that 
$$
\int_0^T\theta^2_t\,dt<\infty,
\quad
\int_0^T|\frH_t|^2_{L_2}\,dt<\infty \,\,\,(a.s.), 
$$
and  for all $(\omega,t)$ we have 
 
\noindent                                    
 (i)  
 $a_t^{ij}$  and $\sigma^i$ (as an $\ell_2$-valued 
function) are continuously differentiable in $x\in\bR^d$ 
for each $i,j=1,2,...,d$ such that 
$$
\sum_{ij}|Da_t^{ij}|_{\bR^d}^2+\sum_{i}|D\sigma^i|^2_{\ell_2(\bR^{d})}
\leq \theta^2_t,
\quad
\text{and}\quad\sum_k\sum_i|\gamma_t^{ki}|_{\bR^d}^2\leq K_1,
$$
for all $x\in \bR^d$;

\noindent                 
(ii) the $\ell_{2}(\bR^d)$-valued function 
 $h_t=h_t(x,u)$ is continuously differentiable in $(x,u)\in\bR^d\times\bR^d$, 
 and the partial derivatives $\partial_x h_t$ and $\partial_u h_t$
take values in $\frl_2:=\ell_{2}(\bR^{d\times d})$ 
such that for all $(x,u)$; 
$$
|\partial_x h_t(x,u)|_{\frl_2}\leq \theta_tu+\frH_t(x), 
\quad
|\partial_u h_t(x,u)|_{\frl_2}\leq K_2;
$$

\noindent
(iii) 
for all $(x,u,z)\in\bR^d\times\bR^d\times\bR^{d\times d}$ 
$$
|f_t(x,u,z)|_{\bR^d}\leq \theta_t |u|_{\bR^d}+K_3|z|_{\bR^{d\times d}}+\frH_t(x).  
\quad
$$                                   
\end{assumption}

\begin{theorem}            
                                                              \label{theorem 3}
Let Assumptions \ref{assumption 1}, \ref{assumption 2},   
\ref{assumption 3} 
and \ref{assumption 4} hold with $\frf=0$. 
Let $u$ be an admissible solution 
to \eqref{eq}-\eqref{ini} such that $u_0\in \cV$ (a.s.). Then 
there is a constant $N=N(r,d)>0$ 
such that if  $\hat u\leq \delta/N$ then $u=(u_t)_{t\in[0,T]}$ 
is a continuous $\cV$-valued $\cF_t$-adapted process such that 
$u\in L_2([0,T], \bW^2_2)$ (a.s.). 
\end{theorem}

To prove this theorem first we prove the following proposition.

\begin{proposition}
                                                          \label{proposition 1.23.2}
Let Assumptions \ref{assumption 1}, \ref{assumption 2}\,(i) and   
\ref{assumption 3}  hold. Let $u$ be an admissible $\bR^d$-valued 
function, such that 
\begin{equation}
                                                                   \label{7.23.2}
u\in L_{\infty}([0,T],\bL_2)\cap L_2([0,T],\bW^1_2)\,(a.s.).
\end{equation}
Then  there is a constant $N=N(d,r)>0$ 
such that 
if $\hat u\leq \delta/N$ then 
\eqref{eq2}-\eqref{ini2} has a unique $\cH$-solution
$v=(v_t)_{t\in[0,T]}$.                                                       
\end{proposition}
\begin{proof}
We are going to cast the above system of equations   
into a stochastic evolution equation 
of the type considered in \cite{GK2024}. 
For this reason set $V:=\cV$, $H:=\cH$, and for 
every $(\omega,t)$ define the linear operators 
$A_t:V\to V^{\ast}$ and 
$B_t=(B^k_t)_{k=1}^{\infty}: V\to \ell_2(H)$ by requiring 
$$
\langle \varphi,A_{t}v\rangle=
-\big(\varphi,u^{Mi}_tD_iv 
+N_{0}\rho_{0}^{-2}\hat u^2v +\delta v\big)-(D_i\varphi, a^{ij}_tD_jv )
$$
\begin{equation}  
                                                                      \label{2.23.2}                                                           
(\varphi,B^k_tv)_H=\big(\varphi,\sigma^{ik}_tD_iv\big), 
\quad k=1,2,...,                                                                                                                              
\end{equation}
where $(\cdot,\cdot)$ and $(\cdot,\cdot)_{H}$ 
are the scalar products in $\bL_{2}$ 
and in $H$, respectively,  
and $N_{0}=N_0(d,r)$ 
is a constant, specified later. 
By Lemma \ref{lemma 12.17.1}
$$
(\varphi,u^{Mi}_tD_iv)\leq |u^{Mi}_t\varphi|_{\bL_2}|D_iv|_{\bL_2} 
\leq N\hat u(|D\varphi|_{L_2}+\rho_0^{-1}|\varphi|_{\bL_2})|v|_{H}  
$$
with a constant $N=N(d,r)$, and 
$$
(D_i\varphi, a^{ij}_tD_jv )\leq N|D\varphi|_{L_2}|Dv|_{L_2}, 
$$
$$
\sum_{k=1}^{\infty}\big|(\varphi,\sigma^{ik}_tD_iv)\big|^2
\leq N^2|\varphi|^2_{H}|Dv|^2_{L_2}
$$
with a constant $N=N(d,\delta)$. Thus $A_t$ and $B_t$ 
are bounded linear operators mapping $V$ into $V^{\ast}$ 
and $\ell_2(H)$, respectively, and their operator 
norms are bounded by $N(d,r,\rho_0,\delta,\hat u)$ 
and $N(d,\delta)$, respectively.  
Moreover, for $v\in V$ 
$$
\langle v,A_{t}v\rangle
\leq -(D_iv, a^{ij}_tD_jv)
-(N_0\rho_{0}^{-2}\hat u^2+\delta)|v|_H^2
$$
\begin{equation}
                                                                 \label{3.23.2}
+(\delta/2)|Dv|^2_{L_2} 
+N_1\hat u |Dv|^2_{L_2}
+N_2\rho_{0}^{-2}\hat u^2|v|^{ 2}_{H}
\end{equation}
with constants $N_1=N_1(d,r)$ and $N_2=N_2(d,r,\delta)$. 
From \eqref{2.23.2} 
$$
|B_tv|^2_{\ell_2(H)}
\leq ||\sigma_t^{i}D_iv|_{\ell_2(\bR^d)}|^2_{L_2}. 
$$
Combining this with \eqref{3.23.2} we get 
$$
 2\langle v, A_{t}v\rangle+|B^{\cdot}_{t}v|^{2}_{\ell_{2}(H)}
 \leq -\big((2a^{ij}-\sigma^{ik}\sigma^{jk})D_jv,D_iv\big)
 $$
 $$
-2(N_0\rho_{0}^{-2}\hat u^2+\delta)|v|_H^2+\delta|Dv|^2_{L_2} 
+2N_1\hat u |Dv|^2_{L_2}
+2N_2\rho_{0}^{-2}\hat u^2|v|^{ 2}_{H}
$$
for $v\in V$.  Hence using Assumption \ref{assumption 1}\,(i) and taking 
$N_0:=N_2$, for $\hat u\leq \delta/(2N_1)$ we have 
\begin{equation}
                                                                                       \label{4.23.2}
 2\langle v, A_{t}v\rangle+|B^{\cdot}_{t}v|^{2}_{\ell_{2}(H)}\leq -\delta |v|^2_{V}  
 \quad\text{for $v\in V$}.                                                                     
\end{equation}
Define for each $(\omega,t)$ the linear operators $\sca_{t}:V\to H$  
and $\scc_t: H\to H$ 
such that 
for all $\varphi\in V$  and $v\in H$ 
\begin{align}
(\varphi,\sca_{t}v)_H
=&(\varphi,-u^{Bi}_tD_iv+\gamma^{ki}(\cR(\sigma_t^{jk}D_jv)^i)     \nonumber\\ 
(\varphi,\scc_tv)_H
=&\big(\varphi,N_{0}\rho_{0}^{-2}v+\delta v\big). \nonumber
\end{align} 
Using the admissibility of $u$, condition \eqref{gamma}  
and that the operator norm 
of $\cR:\bL_2\to\bL_2$ is 1,  we get 
$$
(\varphi,-u^{Bi}_tD_iv+\gamma^{ki}(\cR(\sigma_t^{jk}D_jv)^i)
\leq 
N(\bar u_t+\vartheta_t)|\varphi|_{\bL_2}|Dv|_{L_2}, 
$$
and clearly, 
$$
(\varphi,N_{0}\rho_{0}^{-2}v+\delta v\big)
\leq (N_{0}\rho_{0}^{-2}v+\delta)|\varphi|_{\bL_2}|v|_{\bL_2}.
$$
Hence $\sca_{t}:V\to H$  
and $\scc_t: H\to H$ are bounded operators, and for their 
operator norms, 
$|\sca_t|_{\sca}$ and $|\scc_t|_{\scc}$ we have 
\begin{equation}
                                                                                    \label{5.23.2}
\int_0^T\big(|\sca_t|_{\sca}^2+|\scc_t|_{\scc}\big)\,dt<\infty\,\, (a.s.). 
\end{equation}
Finally for each $(\omega,t)$ we define 
$F^{\ast}_t\in V^{\ast}$, $F_t\in H$ 
and $G_t=(G_t^k)\in\ell_2(H)$ such that 
$$
\langle\varphi, F^{\ast}_t\rangle=(D_i\varphi,\frf_t^i(u_t)), 
\quad
(\varphi,F_t)_{H}=(\varphi,f_t(u_t,Du_t)), 
$$
$$
(\varphi,G^k_t)_{H}=(\varphi,h^k_t(u_t)), 
\quad k=1,2,....,
$$
for $\varphi\in V$. By Assumption \ref{assumption 2}\,(i), 
$$
(D_i\varphi,\frf_t^i(u_t))\leq N|\varphi|_{V}(\lambda_t|u_t|_{\bL_2}
+|\frF_t|_{L_2}),
$$
$$
(\varphi,f_t(u_t,Du_t))
\leq 
N|\varphi|_{H}\big(\kappa_t|u_t|_{\bL_2}+\chi_t|Du_t|_{L_2}+|\frG_t|_{L_2}\big),
$$
with a constant $N=N(d)$, and 
$$
\sum_{k=1}^{\infty}|(\varphi,h^k_t(u_t))|^2
\leq N|\varphi|^2_{\bL_2}(\lambda^2_t|u_t|^2_{\bL_2}
+|\frF_t|^2_{L_2}). 
$$
Hence, due to \eqref{lambda} and \eqref{7.23.2} 
it is easy to check that 
\begin{equation}
                                                                        \label{6.23.2}
\int_0^T\big(|F^{\ast}_t|_{V^{\ast}}^2+|G_t|_{\ell_2(H)}^2
+|F_t|_{H}\big)\,dt<\infty \,\,(a.s.).                                                                       
\end{equation}
Consider now the stochastic evolution equation 
\begin{equation}
                                                                       \label{eeq1}
dv_t=(A_t v_t+\sca_t v_t
+\scc_t v_t+F^{\ast}_t+F_t)\,dt
+(B^k_t v_t+G_t^k)\,dw^k_t, 
\quad t\in(0,T]
\end{equation}
\begin{equation}
                                                                       \label{eini}
v_{t}\big|_{t=0}=u_{0}. 
\end{equation}
Then we can see that $v=(v_t)_{t\in[0,T]}$ 
is a $\cH$-solution 
to \eqref{eq2}-\eqref{ini2} if and only if it is an $H$-solution 
to \eqref{eeq1}-\eqref{eini}. Due to \eqref{4.23.2}, \eqref{5.23.2} 
and \eqref{6.23.2} we can apply Theorem 5.1 in \cite{GK2024}  
to get that 
there is a unique $H$-solution to \eqref{eq2}-\eqref{ini2}, 
which finishes the proof of the proposition. 
\end{proof} 

\begin{proof}[Proof of Theorems \ref{theorem 3}]
Now we set $H=\cV$, and let $V$ be the subspace of divergence free functions 
of $\bW^2_2$. If $h\in\cV$ then $h^{(\varepsilon)}$, the convolution of $h$ with 
$k_{\varepsilon}=\varepsilon^{-d}k(\cdot/\varepsilon)$ for 
a smooth function on $\bR^d$ with compact support and of unit integral, 
belongs to $V$ for every $\varepsilon>0$, and $h^{(\varepsilon)}\to h$ in $H$ 
as $\varepsilon\to0$. This shows that $V$ is a dense subset of $H$. 
It is also separable, as a closed 
subspace of the separable Hilbert space $\bW^2_2$. 
Define for each $(\omega,t)$ the linear operators 
$$
\text{$A_{t}:\,\,V\to V^{*}$
\quad and \quad 
$B_{t}:\,\,V\to \ell_2(H)$}
$$
by requiring
$$
\langle \varphi,A_{t}v\rangle=
-\big((1-\Delta)\varphi,u^{Mi}_tD_iv\big) 
-\big((1-\Delta)\varphi, N_{0}(\rho_{0}^{-2}\hat u^2+1)v\big)
$$
\begin{equation}
                                              \label{1.20.2}
 -\big((1-\Delta)\varphi, \delta v\big)
-(D_i\varphi, a^{ij}_tD_jv )                                             
-(D_kD_i\varphi,a^{ij}_tD_kD_jv), 
\end{equation}
and 
\begin{equation}  
                                                           \label{2.20.2}
(\varphi,B^k_tv)_H=\big(\varphi,\sigma^{ik}_tD_iv\big)                                                                                                                              
+\big(D_l\varphi,\sigma^{ik}_tD_lD_iv\big)
\quad\text{for $k=1,2...,$}  
\end{equation}
to hold for any $\varphi,v\in V$,
where $(\cdot,\cdot)$ and $(\cdot,\cdot)_{H}$ 
are the scalar products in $\bL_{2}$ 
and in $H$, respectively, and $N_{0}=N_0(d,r)$ 
is a constant, specified later. 

Note that by Lemma \ref{lemma 12.17.1}
$$
((\Delta-1)\varphi,u^{Mi}_tD_iv)\leq |\varphi|_{V}|u^{Mi}_tD_iv|_{\bL_2}
\leq N\hat u|\varphi|_{V}(|D^2v|_{L_2}+\rho_0^{-1}|Dv|_{L_2}), 
$$
with $N=N(d,r)$, and clearly,
$$
-(D_i\varphi, a^{ij}_tD_jv )-(D_kD_i\varphi,a^{ij}_tD_kD_jv)
\leq N|\varphi|_{V}|v|_V
$$
with a constant $N=N(d,\delta)$. Hence it is easy to see that 
$A_t$ is a bounded linear operator from 
$V$ into $V^{\ast}$ with operator norm bounded by a constant $N$ 
depending only on $d$,  $r$, $\delta$, $\rho_0$ and $\hat u$. 
Moreover, for $v\in V$ 
$$
\langle v,A_{t}v\rangle
\leq -(D_iv, a^{ij}_tD_jv)
-(D_kD_iv,a^{ij}_tD_kD_jv)
-(N_0\rho_{0}^{-2}\hat u+N_0)|v|_H^2
$$
\begin{equation}
                                                                 \label{2.4.5.24}
-\delta|v|_H^2                                                                 
+(\delta/8)|v|^2_V 
+N_1\hat u |v|^2_{V}
+N_2\rho_{0}^{-2}\hat u^2|v|^{ 2}_{H}
\end{equation}
with constants $N_1=N_1(d,r)$ and $N_2=N_2(d,r,\delta)$.  
Note also that 
$$
\sum_{k=1}^{\infty}|\big(\varphi,\sigma^{ik}_tD_iv\big)                                                                                                                              
+\big(D_l\varphi,\sigma^{ik}_tD_lD_iv)\big)|^2
$$
$$
\leq 2| \varphi|^2_{H}\Big(\sum_{k}|\sigma_t^{ik}D_iv|^2_{\bL_2}
+\sum_{k}|\sigma^{ik}_tD_lD_iv)|^2_{\bL_2}\Big)
$$
$$
\leq N|v|^2_{H}(|Dv|^2_{L_2}+|D^2v|_{L_2})
\leq N|v|^2_{H}|v|^2_{V}
$$
with a constant $N=N(d,\delta)$, which shows that $B_t=(B_t^k)$ 
is bounded linear map from $V$ into $\ell_2(H)$, with operator norm 
bounded by $N$.  From \eqref{2.20.2} we get 
$$
|B_tv|^2_{ H }\leq \sum_{k=1}^{\infty}|\sigma_t^{ik}D_iv|^2_{\bL_2 }
+\sum_{k=1}^{\infty}\sum_{j=1}^d|\sigma_t^{ik}D_jD_iv|^2_{ \bL_2 }+
\tfrac{\delta}{2}|D^2v|^2_{L_2}+N_3|Dv|^2_{L_2}
$$
with a constant $N_3=N_3(d,\delta)$. Combining this with 
\eqref{2.4.5.24} and using that due to Assumption 
\ref{assumption 1} 
$$
-2(D_iv, a^{ij}_tD_jv)
-2(D_kD_iv,a^{ij}_tD_kD_jv)
$$
$$
+
\sum_k|\sigma^{ik}D_iv|^2_{\bL_2 }
+\sum_k\sum_j|\sigma^{ik}D_jD_iv|^2_{ \bL_2 }
\leq -2\delta(|Dv|^2_{L_2}+|D^2v|_{L_2}), 
$$
we obtain 
$$
 2\langle v, A_{t}v\rangle+|B^{\cdot}_{t}v|^{2}_{\ell_{2}(H)}
 \leq -\delta(|Dv|^2_{L_2}+|D^2v|^{2}_{L_2})
 $$
 $$
+2N_2\rho_{0}^{-2}\hat u^2|v|^{ 2}_{H}+
2N_1\hat u|v|^2_V+N_3|v|^2_{H}
 -2N_{0}(\rho_{0}^{-2}
 \hat u^2+1)|v|_{H}^{2}-2\delta |v|^2_{H}
$$ 
Hence for $N_{0}=N_3+N_2$ and
$N_1\hat u\leq \delta/2$
we get
\begin{equation}
                                                                                       \label{10.23.2}
 2\langle v, A_{t}v\rangle+|B^{\cdot}_{t}v|^{2}_{\ell_{2}(
H)}\leq -(\delta/2)|v|^{2}_{V} 
\end{equation}
for $v\in V$. 

Next, we define the linear operators $\sca^{*}_{t}:H\to V^{*}$,
$\scc_t: H\to H$ and 
$\scb_{t}: H\to \ell_2(H)$ 
by requiring that 
for all $\varphi\in V$  and $v\in H$ 
\begin{align}
\langle \varphi,\sca^{*}_{t}v\rangle
=&-((1-\Delta)\varphi,u^{Bi}_tD_iv )_{\bL_{2}}
-(D_kD_i\varphi, D_k a^{ij}_tD_jv)_{\bL_{2}}           \nonumber\\ 
&+\big((1-\Delta)\varphi, \gamma^{ki}(\cR \sigma^{jk}D_jv)^i\big),\nonumber\\
(\varphi,\scb^k_tv)_H
=&(D_l\varphi,D_l\sigma_t^{ik}D_iv)_{\bL_{2}},             \nonumber\\
(\varphi,\scc_tv)_H
=&\big(
(1-\Delta)\varphi,N_{0}(\rho_{0}^{-2}\hat u^2+1)v+\delta v
\big)_{\bL_{2}}                                                              \nonumber
\end{align} 
hold. Due to Assumption \ref{assumption 4} and that $u$ is admissible, 
we have 
$$
((1-\Delta)\varphi,u^{Bi}_tD_iv)
+(D_kD_i\varphi, D_k a^{ij}_tD_jv)_{\bL_{2}}
\leq N(\bar u_t+\theta_t)|\varphi|_{V}|v|_{H},
$$
$$
\Big((1-\Delta)\varphi, \gamma^{ki}(\cR\big(\sigma_t^{jk}D_jv)\big)^i\Big)
\leq
K_1|\varphi|_{V}
\Big(\sum_{k=1}^{\infty}|\cR(\sigma_t^{jk}D_jv)|^2_{\bL_2}\Big)^{1/2}
$$
$$
\leq K_1|\varphi|_{V}\Big(\sum_{k=1}^{\infty}|\sigma_t^{jk}D_jv|^2_{\bL_2}\Big)^{1/2}
\leq 
NK_1|\varphi|_{V}|v|_{H}, 
$$
$$
\sum_{k}|(D_l\varphi,D_l\sigma_t^{ik}D_iv)_{\bL_{2}}|^2
\leq N\theta_t^2|\varphi|^2_{V}|v|^2_{H}
$$
with a constant $N=N(d,\delta)$ and with the constant $K_1$ 
and process $\theta$ from Assumption 
\ref{assumption 4}.  
Clearly, 
$$ 
|\scc_tv|_{H}\leq (N_{0}\rho_{0}^{-2}\hat u^2+N_0+\delta) |v|_H.  
$$
Consequently, $\sca_{t}$, $\sca^{*}_{t}$, 
$\scb_{t}$ and  $\scc_t$ are bounded linear operators
for every $(\omega,t)$, and their operator norms, 
$|\sca^{*}_{t}|$, 
$|\scb_{t}|$ and  $|\scc_t|$, satisfy
\begin{equation}
                                                                                   \label{11.23.2}
\int_0^T\big[|\sca^{*}_{t}|^2+|\scb_{t}|^2+|\scc_{t}|\big]\,dt<\infty\,\, (a.s.). 
\end{equation}
For each $(\omega,t)$ we  define $F^{\ast}_t\in V^{\ast}$, 
$G_t=(G^k_t)\in \ell_2(H)$ by requiring that 
$$
\langle\varphi, F^{\ast}_t\rangle
=((1-\Delta)\varphi, f_t(u_t,Du_t)+\gamma^{ki}_t(\cR h^k_t(u_t))^i), 
$$
\begin{equation*} 
(u,G^k_{t})_{H}=((1-\Delta)\varphi,h^{k}_{t}(u_t))
\end{equation*}
hold for all $\varphi\in V$. By Assumption \ref{assumption 4} 
we have 
$$
((1-\Delta)\varphi, f_t(u_t,Du_t))\leq |\varphi|_V
(\theta_t|u_t|_{\bL_2}+K_3|Du_t|_{L_2}+|\frH_t|_{L_2}), 
$$
$$
((1-\Delta)\varphi, \gamma^{ki}_t(\cR h^k_t(u_t))^i)
\leq |\varphi|_{V}K_1||h_t(u_t)|_{\ell_2(\bR)}|_{L_2}
$$
$$
\leq K_1 |\varphi|_{V}(\lambda_t|u_t|_{\bL_2}+|\frF_t|_{\bL_2}), 
$$
$$
\sum_{k=1}^{\infty}|((1-\Delta)\varphi,h^{k}_{t}(u_t))|^2\
$$
$$
\leq 
2|\varphi|^2_{\bL_2}||h_{t}(u_t)|_{\ell_2(\bR^d)}|^2_{L_2}
+2||D\varphi|_{L_2}|^2|D(h_t(u_t))|_{\ell_2(\bR^{d\times d})}||^2_{L_2}
$$
$$
\leq N|\varphi|^2_H((\lambda_t^2+\theta_t^2)|u_t|^2_{\bL_2}+K_2^2|Du_t|^2_{L_2}
+\frF_t^2+\frH_t^2), 
$$
which show that 
\begin{equation}
                                                                                     \label{12.23.2}
\int_0^T\big(|F^{\ast}_t|^2_{V^{\ast}}+|G_t|^2_{\ell_2}\big)\,dt<\infty \,\, (a.s.).
\end{equation}
Let us now consider the evolution equation 
\begin{equation}
                                                                   \label{8.23.2}
dv_t=(A_tv_t+\sca^{\ast}_tv_t+\scc_tv_t+F_t^{\ast})\,dt 
+(B^k_tv_t+\scb^{k}_t v_t+G^k_t)\,dw^k_t, 
\end{equation}
\begin{equation}
                                                                   \label{9.23.2} 
v_{t}\big|_{t=0}=u_{0}. 
\end{equation}
Then due to \eqref{10.23.2}, \eqref{11.23.2} and 
\eqref{12.23.2} we can apply Theorem 5.1 in 
 \cite{GK2024} to get that \eqref{8.23.2}-\eqref{9.23.2} 
 has a unique $H$-solution $v=(v_t)_{t\in[0,T]}$. Hence it follows 
that almost surely \eqref{1.25.2} holds for $t\in[0,T]$ 
 and $\psi=(1-\Delta)\varphi$ for all divergence-free vector fields 
 $\varphi\in\bW^3_2$, which implies that $v$ is the unique $\cH$-solution 
 of \eqref{eq2}-\eqref{ini2}. By Proposition \ref{proposition 1.23.2} 
there is a constant $N=N(d,r)>0$  
 such that if $\hat u\leq \delta/N$ then $u$ is the unique $\cH$-solution 
 to \eqref{eq2}-\eqref{ini2}. Consequently, there is a 
 (finite) constant 
 $N=N(d,r)>0$ such that  if $\hat u\leq\delta/N$, 
 then almost surely $u_t=v_t$ 
 for all $t\in[0,T]$, that finishes 
 the proof of the theorem. 
 \end{proof}                                    
 \begin{remark}
                                                                  \label{remark Vsol}
We note that if in Definition \ref{definition HL} of a solution $u$ 
to \eqref{eq}-\eqref{ini} we require 
$u\in L_{2}([0,T],\cV)$ (a.s.) instead of 
 $u\in L_{2}([0,T],\cV)\cap L_{\infty}([0,T],\cH)$ (a.s.),  then 
  Theorems \ref{theorem 1},  \ref{theorem 2} and \ref{theorem 3} 
  remain valid, provided in Assumption \ref{assumption 2}  
  we assume that $\lambda$ is a nonnegative constant  and $\kappa$   
  is a nonnegative predictable process such that 
  $$
  \int_0^T\kappa_t^2\,dt<\infty\,\,\text{(a.s.)}, 
  $$ 
  instead of assuming that $\lambda$ and $\kappa$ 
  are nonnegative predictable processes 
  such that 
  $$
  \int_0^T\lambda^2_t+\kappa_t\,dt<\infty\,\,\text{(a.s.)}.
  $$                                           
\end{remark}

\begin{proof}
Notice that in the proof of Theorem \ref{theorem 1} we only use 
$u\in L_{\infty}([0,T],\cH)$ (a.s.) when we prove that 
$$
I:=\int_0^T\lambda_t^2|u_t|^2_{\bL_2}+\kappa_t|u_t|_{\bL_2}
+\vartheta_t\lambda_t|u_t|_{\bL_2}\,dt<\infty\,\,\text{(a.s.)}, 
$$
see \eqref{esssup 1}, \eqref{esssup 2} and \eqref{esssup 3}. Clearly, 
if $\lambda$ is a nonnegative constant and $\kappa\in L_2([0, T],\bR)$ (a.s.)
then 
$$
I\leq \lambda^2\int_0^T|u_t|^2_{\bL_2}\,dt
+\Big(\int_0^T\kappa^2_t\,dt\Big)^{1/2}\Big(\int_0^T|u_t|^2_{\bL_2}\,dt\Big)^{1/2}
$$
$$
+\lambda\Big(\int_0^T\vartheta_t^2\,dt\Big)^{1/2}
\Big(\int_0^T|u_t|^2_{\bL_2}\,dt\Big)^{1/2}<\infty\,\, \text{(a.s.)}, 
$$
which proves the remark. 
\end{proof}

\end{document}